\documentclass[11pt,leqno]{amsart}

\usepackage[utf8x]{inputenc} 	      
\usepackage{lmodern}                
\usepackage[T1]{fontenc}            
\usepackage[english]{babel}  
\usepackage{ucs}             	      
\usepackage[title,titletoc]{appendix} 

\usepackage{geometry}
\usepackage{graphicx}     
\usepackage{xcolor}       
\usepackage{pstricks-add} 
\usepackage{enumitem}     
\usepackage{verbatim}     

\usepackage{amsmath, amsthm}
\usepackage{amsfonts, amssymb}
\usepackage{bm}               

\usepackage{hyperref}
\hypersetup{
	pdfstartview={FitH},    
  pdfauthor={Thomas Rey}, 
	linkcolor=blue,         
	citecolor=magenta,      
}

\theoremstyle{plain}
	\newtheorem{theorem}{Theorem}[section]
	\newtheorem*{theorem*}{Theorem}
	\newtheorem{lemma}{Lemma}[section]
	\newtheorem{proposition}{Proposition}[section]
	\newtheorem*{proposition*}{Proposition}
	\newtheorem{corollary}{Corollary}[section]

\theoremstyle{remark}
	\newtheorem{remark}{\textbf{Remark}}[]

\theoremstyle{definition}
	\newtheorem{definition}{Definition}[]

\DeclareMathOperator{\vect}{Span}

\DeclareMathOperator{\range}{R}
\DeclareMathOperator{\id}{Id}
\DeclareMathOperator{\domain}{dom}
\DeclareMathOperator{\codim}{codim}
\DeclareMathOperator{\deficiency}{def}
\DeclareMathOperator{\nullity}{nul}
\DeclareMathOperator{\supess}{supess}

\newcommand{\ve}{\varepsilon}
\newcommand{\eg}{\emph{e.g. }}

\newcommand{\RR}{\mathbb{R}}
\newcommand{\NN}{\mathbb{N}}

\newcommand{\TT}{\mathbb{T}}
\newcommand{\CC}{\mathbb{C}}
\newcommand{\SSS}{\mathbb{S}}
\newcommand{\Q}{\mathcal{Q}}

\newcommand{\LL}{\mathcal{L}}

\newcommand{\squaredBox}[1]{
  \medskip
  \begin{quote}
	  \fbox{\parbox{\linewidth\fboxrule\fboxsep}{#1}}
	\end{quote}
	\medskip
}

\author{Thomas Rey}
\address{Thomas Rey \\
CSCAMM, The University of Maryland \\
CSIC Building, Paint Branch Drive \\
College Park, MD 20740  \\
USA}
\email{trey@cscamm.umd.edu }

	

\title[Spectral Study of the Linearized Granular Gases Operator]{A Spectral Study of the Linearized Boltzmann Equation for Diffusively Excited Granular Media}

\date{}

\begin{document}

  \begin{abstract}
    In this work, we are interested in the spectrum of the diffusively excited granular gases equation, in a space inhomogeneous setting, linearized around an homogeneous equilibrium.
    
    We perform a study which generalizes to a non-hilbertian setting and to the inelastic case the seminal work of Ellis and Pinsky \cite{ellis:1975} about the spectrum of the linearized Boltzmann operator. 
    We first give a precise localization of the spectrum, which consists in an essential part lying on the left of the imaginary axis and a discrete spectrum, which is also of nonnegative real part for small values of the inelasticity parameter.
    We then give the so-called inelastic ``dispersion relations'', and compute an expansion of the branches of eigenvalues of the linear operator, for small Fourier (in space) frequencies and small inelasticity.
    
    One of the main novelty in this work, apart from the study of the inelastic case, is that we consider an exponentially weighted  $L^1(m^{-1})$ Banach setting instead of the classical $L^2(\mathcal M_{1,0,1}^{-1})$ Hilbertian case, endorsed with Gaussian weights. We prove in particular that the results of \cite{ellis:1975} holds also in this space.

  \end{abstract}
  
	\maketitle
	
	\tableofcontents
	
	\newpage

	\section{Introduction}
	  \label{secIntro}
		
		Let $f^\ve := f^\ve(t,x,v)$ be a solution to the space inhomogeneous collisional kinetic equation
		\begin{equation}
		  \label{eqBoltzEPS}
			\frac{\partial f^\ve}{\partial t} + v \cdot \nabla_x f^\ve = \frac{1}{\varepsilon}\left ( \mathcal Q_\alpha(f^\ve,f^\ve) + \ve \, \Delta_v (f^\ve) \right ),
		\end{equation}
		where $t \geq 0$, $v \in \RR^d$ and $x\in \Omega$, for $\Omega$ being either the whole space domain $\RR^d$ or the torus\footnote{The case of a square domain $[-L,L]^d$, for $L \geq 0$ with specular reflection on the boundary can also be seen as a particular case of a torus made of $2^d$ independent copies of the initial box, using the parity of the normal component of the velocity of $f^\ve$ at the boundary (as noticed by Grad in \cite{Grad:1964}).} $\TT^d$.
		The collision operator $\Q_\alpha$ is the so-called \emph{granular gases} operator (sometimes known as the inelastic Boltzmann operator), describing an energy-dissipative microscopic collision dynamics, which we will present in the following section.
		The parameter $\ve > 0$ is the scaled \emph{Knudsen} number, that is the ratio between the mean free path of particles before a collision and the length scale of observation.
		
		Once $\ve$ goes to $0$, and then when the number of collisions per time unit goes to infinity, the complexity of equation \eqref{eqBoltzEPS} is (formally) greatly reduced, the solution being described almost completely by its local hydrodynamic fields, namely its \emph{mass} $N \geq 0$, its \emph{momentum}    $\bm u \in \RR^d$ and its \emph{temperature} $T \geq 0$. 
		These quantities are obtained from a particle distribution function $f$ by computing the first moments in velocity:
		\begin{equation}
      \label{defMacroQuantities}
      \begin{gathered}
	      N(t,x) \, = \, \int_{\RR^d} f(t,x,v) \, dv, \qquad 
	      N(t,x) \, \bm u(t,x) \, = \, \int_{\RR^d} f(t,x,v) \, v \, dv, \\
	      T(t,x) \, = \, \frac{1}{d \, N} \int_{\RR^d} f(t,x,v) \, |v - \bm u|^2 \, dv .
	    \end{gathered}
    \end{equation}  
    This reduction is usually carried on using the so-called Hilbert or Chapman-Enskog expansions of the solutions to a linearized version of the kinetic equation \eqref{eqBoltzEPS} (see \eg the book of Cercignani, Illner and Pulvirenti \cite{CIP:94} for a complete mathematical introduction in the elastic case).
    
    A rigorous mathematical proof of this ``contraction of the kinetic description'' (namely the hydrodynamic limit of the kinetic model towards a macroscopic one) for the elastic case has been first given for the linear setting in the paper of Ellis and Pinsky \cite{ellis:1975} but the inelastic case still remains to be investigated.
    An important step in the proof of this elastic limit is to give the so-called \emph{dispersion relations} of the collision operator, namely a Taylor expansion of the eigenvalues of the linearization of the collision operator, with respect to the space variable, near a global equilibrium (and this was the main purpose of \cite{ellis:1975}). 
    The precise knowledge of the dispersion relations is actually of crucial interest in the study of the full nonlinear and compressible hydrodynamic limit and it was for example used by Kawashima, Matsumura and Nishida in \cite{Nishida:78,Kawashima:79} (as a part of a rather abstract Cauchy-Kowalevski-type argument which is also related to Niremberg \cite{Nishida:77}).
    The work of Caflisch \cite{Caflisch:80} also relies (but perhaps not as centrally as the previous ones) on these dispersion relations.
    Let us also quote the work of Degond and Lemou \cite{DegondLemou:1997} where a similar analysis of the dispersion relations was conducted for the linearized Fokker-Planck equation.
    
    We propose to give in this paper the corresponding inelastic expansion, with respect to both the space variable and the inelasticity parameter, allowing to investigate in a future work first the two linearized hydrodynamic limits of our model ``à la Ellis et Pinsky'' and then the nonlinear, compressible ones ``à la Nishida''. 
    This result will allow us in particular to confirm a claim concerning the  \emph{clustering} behavior of granular gases made in the classical textbook \cite[p. 238]{brilliantov:2004} after a formal analysis, namely that 
    \begin{quote}
      \emph{the smaller the inelasticity, the larger the system must be to reveal clusters.}
    \end{quote}

	  \subsection{The Model Considered}

			Let $\alpha \in (0,1]$ be the restitution coefficient of the microscopic collision process, that is the ratio of kinetic energy dissipated during a collision, in the direction of impact. 
			Then, we can define a strong form of the \emph{collision operator} $\mathcal Q_\alpha$ by
			\begin{align}
			  \mathcal Q_\alpha(f, g)(v) & = \int_{\mathbb{R}^d \times \mathbb{S}^{d-1}} |u| \left( \frac{\, 'f \, 'g_*}{\alpha^2} - f \,g_* \right) b(\widehat u \cdot \sigma) \, d \sigma \, dv_*, \label{defBoltzOp}\\
			   & = \mathcal Q_\alpha^+(f,g)(v) - f(v) L(g)(v) \notag,
			\end{align}
			where we have used the usual shorthand notation $\, 'f := f('v)$, $\, 'f_* := f('v_*)$, $f := f(v)$, $f_* := f(v_*)$ and $\widehat u := u/|u|$. In \eqref{defBoltzOp}, $\, 'v$ and $\, 'v_*$ are the pre-collisional velocities of two particles of given velocities $v$ and $v_*$, defined for $\sigma \in \SSS^{d-1}$ as
			\begin{equation*}
				\left\{\begin{aligned} 
					& 'v = \frac{v + v_*}{2} - \frac{1 - \alpha}{4 \, \alpha} (v-v_*) + \frac{1+ \alpha}{4 \, \alpha} |v - v_*| \, \sigma, \\
					& 'v_* = \frac{v + v_*}{2} + \frac{1 - \alpha}{4 \, \alpha} (v-v_*) - \frac{1+ \alpha}{4 \, \alpha} |v - v_*| \, \sigma.
				\end{aligned} \right.
			\end{equation*}
			The unitary vector $\sigma$ is the center of  the \emph{collision sphere}  (see Figure \ref{figCollSphereInel}) and $u := v - v_*$  is the \emph{relative velocity} of the pair of particles.
			Finally, the function $b$ is the so-called \emph{angular cross-section}, describing the probability of collision between two particles. We assume that
			\begin{equation}
			  \label{hypCrossSec1}
			  b \text{ is a Lipschitz, non-decreasing and convex function on } (-1,1),
			\end{equation}
			and also that it is bounded from above and below by two nonnegative constants $b_m$ and $b_M$:
			\begin{equation}
			  \label{hypCrossSec2}
			  b_m \leq b(x) \leq b_M, \quad \forall x \in (-1, 1).
			\end{equation}
			In particular, this cross-section is integrable on the unit sphere, thus fulfilling the so-called Grad's cut-off assumption\footnote{Physically relevant in the case of inelastic collisions, due to the macroscopic size of the grains forming the gas.}.
		  The operator $\mathcal Q_\alpha^+(f,g)(v)$ is usually known as the \emph{gain} term because it can be understood as the number of particles of velocity $v$ created by collisions of particles of pre-collisional velocities $\, 'v$ and $\, 'v_*$, whereas $f(v)L(g)(v)$ is the \emph{loss} term, modeling the loss of particles of pre-collisional velocities $\, 'v$. 
		  
		  We can also give a weak form of the collision operator. 
		  Indeed, if $\omega \in \SSS^{d-1}$ is the direction of impact, we can parametrize the post-collisional velocities $v'$ and $v_*'$ as
		  \begin{equation*}
				\left \{\begin{aligned} 
					v' & = v - \frac{1 + \alpha}{2} \left( u \cdot \omega \right) \omega, \\
					v_*' & = v_* + \frac{1 + \alpha}{2} \left( u \cdot \omega \right) \omega.
				\end{aligned}\right .
			\end{equation*}
			Then we have the weak representation, for any smooth test function $\psi$,
			\begin{equation} 
			  \label{defBoltzweak}
				\int_{\RR^d} Q_\alpha(f,g) \, \psi(v) \, dv = \frac{1}{2 }\int_{\mathbb{R}^d \times \mathbb{R}^d \times \mathbb{S}^{d-1}} |u|f_{*} \, g \, \left(\psi' + \psi_*' - \psi - \psi_* \right) b(\widehat u \cdot \omega ) \, d\omega \, dv \, dv_*.
			\end{equation}
			
			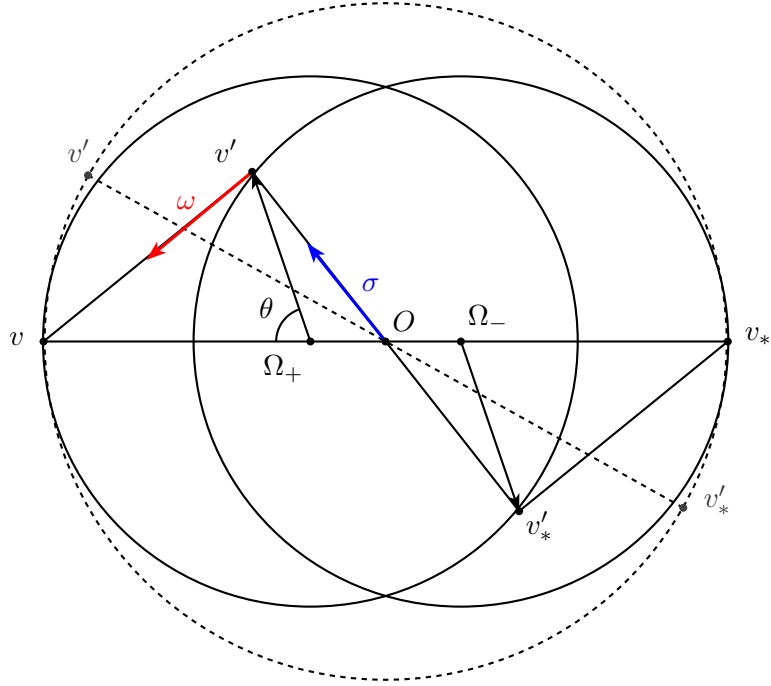
\begin{figure}
				\begin{center}
					\psset{xunit=4.5cm,yunit=4.5cm,algebraic=true,dotstyle=o,dotsize=3pt 0,linewidth=0.8pt,arrowsize=3pt 2,arrowinset=0.25}
					\begin{pspicture*}(-1.11,-1.05)(1.12,1.08)
						\pscircle(0.22,0){3.53}
						\psline(-0.39,0.5)(0.39,-0.5)
						\psline(-1,0)(1,0)
						\pscircle[linestyle=dashed,dash=2pt 2pt](0,0){4.5}
						\pscircle(-0.22,0){3.53}
						\psline[linestyle=dashed,dash=2pt 2pt](-0.87,0.49)(0.87,-0.49)
						\psline{->}(-0.22,0)(-0.39,0.5)
						\psline{->}(0.22,0)(0.39,-0.5)
						\parametricplot{1.906}{3.1416}{0.1*cos(t)+-0.22|0.1*sin(t)+0}
						\psline[linewidth=1.2pt,linecolor=blue]{->}(0,0)(-0.23,0.29)
						\psline(-1,0)(-0.39,0.5)
						\psline(0.39,-0.5)(1,0)
						\psline[linewidth=1.2pt,linecolor=red]{->}(-0.39,0.5)(-0.7,0.24)
						\psdots[dotstyle=*](-1,0)
						\rput[bl](-1.1,-0.01){$v$}
						\psdots[dotstyle=*](1,0)
						\rput[bl](1.05,-0.01){$v_*$}
						\psdots[dotstyle=*](0,0)
						\rput[bl](0.02,0.03){{$O$}}
						\psdots[dotstyle=*](-0.22,0)
						\rput[tr](-0.24,-0.04){$\Omega_+$}
						\psdots[dotstyle=*](0.22,0)
						\rput[bl](0.24,0.03){$\Omega_-$}
						\psdots[dotstyle=*](-0.39,0.5)
						\rput[bl](-0.5,0.53){$v'$}
						\psdots[dotstyle=*](0.39,-0.5)
						\rput[tl](0.41,-0.5){$v_*'$}
						\psdots[dotstyle=*,linecolor=darkgray](-0.87,0.49)
						\rput[bl](-0.93,0.53){\darkgray{$v'$}}
						\psdots[dotstyle=*,linecolor=darkgray](0.87,-0.49)
						\rput[bl](0.93,-0.51){\darkgray{$v_*'$}}
						\rput[bl](-0.37,0.07){$\theta$}
						\rput[bl](-0.07,0.15){\blue{$\sigma$}}
						\rput[bl](-0.61,0.39){\red{$\omega$}}
					\end{pspicture*}
					\caption{Geometry of inelastic collisions, $O := (v+v_*)/2 $ and $\Omega_{\pm} := O \pm (v_* - v)\, (1-e)/2 $ (dashed lines represent the elastic case).}
					\label{figCollSphereInel}
				\end{center}
			\end{figure}

			Thanks to this expression, we can compute the macroscopic properties of the collision operator $\Q_\alpha$. Indeed, we have the microscopic conservation of impulsion and dissipation of kinetic energy:
			\begin{align*}
				v' + v_*' & = v + v_*, \\
				|v'|^2 + |v_*'|^2 - |v|^2 - |v_*|^2 & = - \frac{1-\alpha^2}{2} | u \cdot \omega |^2 \leq 0.
			\end{align*}
			Then if we integrate the collision operator against $\varphi(v) = (1, \, v\, |v|^2)$, we obtain the preservation of mass and momentum and the dissipation of kinetic energy:
			\begin{equation*}
		    \int_{\RR^d}  \Q_\alpha(f,f)(v) \begin{pmatrix} 1 \\ v \\ |v|^2 \end{pmatrix} dv \,=\, \begin{pmatrix} 0 \\ 0 \\ - (1-\alpha^2) D(f,f) \end{pmatrix},
		  \end{equation*}
		  where $D(f,f) \geq 0$ is the \emph{energy dissipation} functional, given by 
		  \begin{equation}
		  	\label{defDissipFunctional}
		    D(f,f) := b_1 \int_{\mathbb{R}^d \times \mathbb{R}^d}f \, f_* \, |v-v_*|^3 \, dv \, dv_* \geq 0,
		  \end{equation}
		  and $b_1$ is the angular momentum, depending on the cross-section $b$ and given by
		  \begin{equation*}
		    b_1 := \int_{\SSS^{d-1}} (1 - (\widehat u \cdot \omega)) \, b(\widehat u \cdot \omega ) \, d \omega < \infty.
		  \end{equation*}
		  It is of course finite thanks to the bounds \eqref{hypCrossSec2}.
		  
		  In all the following of the paper, we shall assume that the restitution coefficient is related to the Knudsen number in the following way :
		  \[ \alpha = 1 - \ve.\]
		  The macroscopic properties of the collision operator, together with the conservation of positiveness, imply that the equilibrium profiles of $\Q_\alpha$ are trivial Dirac masses (see \eg the review paper \cite{Villani:2006} of Villani). Nevertheless, adding a thermal bath $(1-\alpha)\Delta_v$ will prevent this fact.
		  Indeed, the existence of a non-trivial equilibrium profile $F_\alpha$ to the space homogeneous granular gases equation with a thermal bath is insured by the competition occurring between the dissipation of kinetic energy occasioned by the collision operator $\Q_\alpha$ and the gain of energy given by the diffusion term $\Delta_v$.
		  
		  More precisely, if we multiply the equation $\mathcal{Q}_\alpha (f, f) + (1-\alpha) \, \Delta_v(f) = 0$ by $|v|^2$, integrate in velocity and divide by $1-\alpha$, we obtain using \eqref{defDissipFunctional} the balance equation
		  \begin{equation}
		    \label{defBalanceEq}
		    \left (1 + \alpha\right ) D(f,f) = 2 \, d.
		  \end{equation}	  
		  It has then been shown in \cite{bobylev:2004,mischler:20092} that under the hypotheses \eqref{hypCrossSec1}--\eqref{hypCrossSec2} on the cross-section, there exists $\alpha_* \in (0,1)$ such that for all $\alpha \in [\alpha_*, 1]$, there exists an unique equilibrium profile $0 \leq F_\alpha \in \mathcal S(\RR^d)$ of unit mass and zero momentum:
		  \begin{equation} 
		    \label{eqEquilib}
		    \left \{ \begin{aligned}
		      & \mathcal{Q}_\alpha (F_\alpha, F_\alpha) + (1-\alpha) \, \Delta_v(F_\alpha) = 0, \\
		      & \, \\
		      & \int_{\RR^d} F_\alpha(v) \, dv = 1, \qquad \int_{\RR^d} F_\alpha(v) \, v \, dv =  0.
		    \end{aligned} \right.
		  \end{equation}
		  In the last expression, $\mathcal S(\RR^d)$ denotes the Schwartz class of $\mathcal C^\infty$ functions decreasing at infinity faster than any polynomials. 
		  The tails of this distribution are exponentials, of order $3/2$.
		  
		  Of course, if $\alpha = 1$ (elastic, non-heated case), the distribution $F_1$ is nothing but the following \emph{Maxwellian}\footnote{Hence, there is a bifurcation which occur between the inelastic heated case and the elastic nonheated one.} distribution 
		  \begin{equation}
		    \label{defG1}
		    F_1(v) := \mathcal {M}_{1,0,\bar T_1}(v),
		  \end{equation}
		  where $\mathcal{M}_{N,\,\bm u,\,T}$ is the Maxwellian distribution of mass $N$, velocity $\bm u$ and temperature $T$, only equilibria of the elastic collision operator $\Q_1$ (see \eg \cite{CIP:94} for more details), and given by 
		  \[ \mathcal{M}_{N,\,\bm u,\,T}(v) := \frac{N}{(2 \pi T )^{d/2}} \exp\left(\frac{|v - \bm u|^2}{2 T} \right )\]
		  for $(N, \bm{u}, T) \in \RR^{d+2}$. The quantity $\bar T_1$ in \eqref{defG1} is defined by  passing to the limit $\alpha \to 1$ in the balance equation \eqref{defBalanceEq} :
		  \begin{equation*}
		    D(F_1,F_1) =  d.
		  \end{equation*}
		  We can then show thanks to this relation (see \cite{mischler:20092} for details) that $\bar T_1$ is given by
		  \begin{equation}
		    \label{defElasticEquilibTemperature}
		    \bar T_1 = \frac{1}{2} \frac{d^{2/3}}{b_1^{2/3}} \left ( \int_{\RR^d} \mathcal{M}_{1,0,1}(v) |v|^3 \, dv\right )^{-2/3}.
		  \end{equation}
		  
    \subsection{The Linearized Operator}
				
			As we have said in the introduction, our goal is to perform the fluid dynamic limit $\ve \to 0$ of equation \eqref{eqBoltzEPS}. 
			By rescaling the time $\widetilde t = t / \ve$ and introducing a new distribution $\widetilde f(\widetilde t, x, v) = f(t,x,v)$, the equation \eqref{eqBoltzEPS} now reads (forgetting the tildas)
			\begin{equation}
		    \label{eqBoltzEPSscal}
			  \frac{\partial f^\ve}{\partial t} + \ve \, v \cdot \nabla_x f^\ve = \mathcal Q_\ve(f^\ve,f^\ve) + \ve \, \Delta_v (f^\ve).
		  \end{equation}
			The hydrodynamic limit then amounts to consider the large time, small space variations of the model (see the paper of Carlen, Chow and Grigo \cite{carlen:2010} for more details on the scaling and on the different types of limit models it can yields).
			This means as $\ve = 1-\alpha$ that we are studying \emph{fluctuations} $g$ of $f^\alpha$ near the space homogeneous equilibrium profile $F_\alpha$: 
			\begin{equation}
			  \label{defFluc}
			  f^{\alpha} = F_\alpha + g. 
			\end{equation}
			By plugging this expansion on equation \eqref{eqBoltzEPSscal} and using the equilibrium relation \eqref{eqEquilib}, we obtain the following equation for $g$:
			\begin{equation}
			  \label{eqScalFluc}
			  \frac{\partial g}{\partial t} + (1-\alpha) \, v \cdot \nabla_x g = \LL_\alpha \, g + (1-\alpha) \, \Gamma_\alpha(g,g),
			\end{equation}
      where the linearized operator $\mathcal{L}_\alpha$ is given for $v \in \RR^d$ by
			\begin{equation*}
	      \mathcal{L}_\alpha(g)(v) \, := \, \mathcal{Q}_\alpha (g, F_\alpha)(v) +  \mathcal{Q}_\alpha (F_\alpha, g)(v) + (1-\alpha) \Delta_v(g)(v),
			\end{equation*}
			and $\Gamma_\alpha$ is the quadratic remainder.
			
			In order to prove rigorous results on the original model, such as nonlinear stability, it will be crucial that the fluctuation $g$ lives in a weighted $L^1$ space.
			Indeed, to this purpose, we shall need to connect the properties of the linearized operator $\LL_\alpha$ to the existing $L^1_3$ \emph{a priori} estimates for the nonlinear operator $\Q_\alpha$. These regularity properties were discussed extensively by  Mischler and Mouhot  in the series of paper \cite{mischler:20091,mischler:20092}.
			As we can see in these papers (we recalled the most important properties in the Appendix), we will need to take $g \in L^1(m^{-1} )$, for $m$ an \emph{exponential weight} function: there exists $a>0$ and $0 < s < 1$ such that
			\begin{equation}
			  \label{defExpWeightFunc}
			  m(v) := \exp(-a\,|v|^s).
		  \end{equation}
		  The expansion \eqref{defFluc} is well defined provided that the original distribution $f \in L^{1}\left (m^{-1}\right )$.

			Let us present some basic properties of the linear operator $\LL_\alpha$. 
			We first need to define the so-called \emph{collision frequency} by
	    \begin{equation*}
	      \nu_\alpha(v) := L(F_\alpha)(v) = \int_{\mathbb{R}^d \times \mathbb{S}^{d-1}} |v-v_*| \, F_\alpha(v_*) \, b(\widehat u \cdot \sigma) \, d \sigma \, dv_*.
	    \end{equation*}
	    It is known (see for example the lemma 2.3 of  \cite{mischler:20091} for an elementary proof) that for any $g \in L^1_3(\RR^d$), there exists some explicit nonnegative constants $c_0$, $c_1$ such that
	    \[0 < c_0 \, (1 + |v|) \leq L(g)(v) \leq c_1 \, (1+|v|), \quad \forall \in v \in \RR^d. \]
	    In particular, the collision frequency $\nu_\alpha$ verifies
			\begin{equation} 
			  \label{nuCrois}
				0 < \nu_{0,\alpha} \, (1 + |v|) \leq \nu_\alpha(v) \leq \nu_{1,\alpha} \, (1+|v|),
			\end{equation}
			for two explicit nonnegative constants $\nu_{0,\alpha}$, $\nu_{1,\alpha}$.			
			Then, we can rewrite the linearized collision operator as a difference of nonlocal and local operators:
			\begin{equation*}
			  \LL_\alpha (g)\, = \, \mathcal L_\alpha^+(g) - \mathcal L^*(g) - \mathcal L^{\nu_\alpha}(g),
			\end{equation*}
			where  $\mathcal L_\alpha^+$ is the linearization near $F_\alpha$ of the gain term, $\mathcal L^*$ a convolution operator and  $\mathcal L^\nu$ is the operator of multiplication by a function of the velocity variable $\nu$. 
			Classically, for $\alpha = 1$, the linearized operator splits between a compact operator  on $L^1(m^{-1})$ (see the paper of Mouhot \cite{Mouhot:2006} for this particular exponentially weighted $L^1$ case) and a multiplication operator:
			\begin{align*}
			  \mathcal L_1(g) & \, = \, \mathcal L_1^c(g) - \mathcal L^{\nu_1}(g).
	    \end{align*}		 
	    We will see in Section \ref{secLocaliz} that the same type of decomposition holds for $\LL_\alpha$.

			As a first step to treat mathematically the question of the hydrodynamic limit of equation \eqref{eqBoltzEPS}, we shall forget the nonlinearity in equation \eqref{eqScalFluc} and study the hydrodynamic limit of the linear equation
			\begin{equation}
			  \label{eqScalFlucLin}
			  \frac{\partial g}{\partial t} + (1-\alpha) \, v \cdot \nabla_x g = \LL_\alpha \, g.
			\end{equation}
			One strategy of proof is to compare the spectrum of    the linear operator
	  	\begin{equation}
			  \label{defOpLinXV}
				- (1-\alpha) \, v\cdot \nabla_x + \mathcal{L}_\alpha,
			\end{equation}
			to the one of the linearized fluid equation associated to the limit, as done in the seminal paper of Ellis and Pinsky \cite{ellis:1975}.
			As a byproduct, the study of this spectrum will allow us to answer to the question of the stability of the solutions to equation \eqref{eqScalFlucLin}, by proving that the real part of the eigenvalues of \eqref{defOpLinXV}	remains nonpositive. 
			Hence, the rest of this paper is devoted to the computation of the spectrum (for small inelasticity and small space positions) of \eqref{defOpLinXV}.

			In order to avoid to deal with the free transport operator in differential form, we shall now use Fourier transform in space.
			More precisely, if we define the Fourier transform in $x$ of a function $ \varphi : \RR^d \to \RR$ as
			\[ \mathcal{F}_x(\varphi)(\xi) := \int_{\RR^d} e^{-i \xi \,\cdot \,x} \varphi(x) \, dx, \quad \forall \, \xi \in \RR^d, \]
			it is well know that
			\[ \mathcal F_x\left (\nabla g\right )(\xi) = i \, \xi \, \mathcal{F}_x(g)(\xi) .\]
			Then using the fact that  $\mathcal{L}_\alpha$ only acts on velocity variables and setting 
			\[ \gamma := \left (1 - \alpha\right ) \xi, \]
			we can write   \eqref{defOpLinXV} in (scaled) spatial Fourier variables as
			\begin{equation} 
			  \label{eqOpLinFour}
				- i \, (\gamma \cdot v) + \mathcal{L}_\alpha =: \LL_{(\alpha, \, \gamma)}.
			\end{equation}
			This operator is well defined on $L^1\left (m^{-1}\right )$, with domain $\domain(\LL_{\alpha, \, \gamma})= W_1^{2,1}\left (m^{-1}\right )$. 
			In this particular set of variable, the equation \eqref{eqScalFlucLin} finally reads
			\[ \frac{\partial g}{\partial t} = \LL_{\alpha, \, \gamma} \, g,\]
			and we see now the need to study the spectrum of the linear operator $\LL_{\alpha, \, \gamma}$ for small values of the variable $\gamma$.
			
			To finish with the definitions, let us denote by $N_1$ the kernel of the elastic operator $\mathcal{L}_1$. It is spanned by the elastic \emph{collisional invariants}, namely 
			\[  N_1 := \vect \{ F_1, \, v_i \, F_1 , |v|^2 \, F_1 : \, 1 \leq i \leq d \},\]
			where $F_1 = \mathcal M_{1,0,\bar T_1}$ and $T_1$ is the quasi-elastic equilibrium temperature \eqref{defElasticEquilibTemperature}.
			For $\alpha < 1$, the kernel $N_\alpha$ of the inelastic operator $\LL_\alpha$ is smaller, because of the lack of energy conservation; it is given by
			\[  N_\alpha := \vect \{F_\alpha, \, v_i \, F_\alpha : \, 1 \leq i \leq d \}.\]
			
	  \subsection{Functional Framework and Main Results}
			
			Let us present some functional spaces needed in the paper. We denote by $L_q^p$ for $p \in [1, +\infty)$ and $q \in [1, +\infty)$ the following weighted Lebesgue spaces:
			\begin{equation*}
				L^p_q = \left\{ f : \mathbb{R}^d \rightarrow \mathbb{R} \text{ measurable; }\|f\|_{L^p_q} := \int_{\mathbb{R}^d} |f(v)|^p \, \langle v\rangle^{p q} \, dv < \infty \right\},
			\end{equation*}
			where $\langle v\rangle := \sqrt{1+ |v|^2}$. The weighted $L^\infty_q$ is defined thanks to the norm
			\begin{equation*}
				\|f\|_{L^\infty_q} := \supess_{v \in \RR^d} \left ( |f(v)|  \, \langle v\rangle^{q}\right ).
			\end{equation*}
			Then, we denote for $s \in \NN$ by $W_q^{s,p}$ the weighted Sobolev space
			\begin{equation*}
				W_q^{s,p} := \left\{ f \in L^p_q; \|f\|_{W_q^{s,p}}^p := \sum_{|k| \leq s} \int_{\RR^d} \left |\partial^k f(v)\right |^p  \langle v\rangle^{p \,q} \, dv < \infty \right \}.
			\end{equation*}
		  The case $p=2$ is the Sobolev space $H^s_q := W^{s,2}_q$, which can also be defined thanks to Fourier transform by the norm
			\begin{equation*}
				\|f\|_{H_q^{s}}^2 := \left \| \mathcal F_v \left (f \, \langle \cdot \rangle^s \right )\right \|_{L^2_q}.
			\end{equation*}	
      We also need to define the more general weighted spaces $L^p(m^{-1} )$ and $W^{s,p}(m^{-1} )$, where $m$ is an exponential weight function given by \eqref{defExpWeightFunc} respectively by the norms
			\begin{gather*}
			  \|f\|_{L^p(m^{-1})}^p := \int_{\mathbb{R}^d} |f(v)|^p \, m^{-1}(v) \, dv, \\
			  \|f\|_{W^{s,p}(m^{-1})}^p := \sum_{|k| \leq s} \left \|\partial^k f\right \|_{L^p(m^{-1})}^p. 
      \end{gather*}	
      
      For the sake of completeness, let us finally state some notions about operators that we shall need in the following.
	    
	    \begin{definition}
				A \emph{closed} operator  $T$  defined on a Banach space $ X$ is said to be a
				\begin{itemize}
					\item \emph{Fredholm} operator of index $(\nullity(T),\deficiency(T))$ if the quantities $\nullity(T) := \dim(\ker T)$ (the \emph{nullity})  and $\deficiency(T) := \codim(\range(T))$ (the \emph{deficiency}) are finite;
					\item \emph{semi-Fredholm} operator if $\range(T)$ is closed and at least one of these two quantities are finite.
				\end{itemize}
				For such an operator, we define the
				\begin{itemize}
					\item \emph{resolvent set} $R(T) \subset \mathbb{C}$ and the \emph{resolvent operator} $\mathcal{R}(T, \zeta)$ as
						\begin{equation*}
							 R(T) := \{ \zeta \in \mathbb{C} : T - \zeta \text{ is invertible on $X$, of bounded inverse } \mathcal{R}(T, \zeta) \};
						\end{equation*}
					\item \emph{spectrum} $\Sigma(T)$ of $T$ as the (closed) set \[\Sigma(T) := R(T)^c;\]
					\item \emph{Fredholm} set $\mathcal{F}(T) \subset \mathbb{C}$  of $T$ as
						\begin{equation*}
							\mathcal{F}(T) := \{\zeta \in \mathbb{C} : T - \zeta \text{ is Freholm} \};
						\end{equation*}
					\item \emph{semi-Fredholm} set $\mathcal{SF}(T) \subset \mathbb{C}$  of $T$ as
						\begin{equation*}
							\mathcal{SF}(T) := \{\zeta \in \mathbb{C} : T - \zeta \text{ is semi-Freholm} \};
						\end{equation*}
					\item \emph{essential spectrum} $\Sigma_{ess}(T)$ of $T$ as the set  \[\Sigma_{ess}(T) := \mathcal{SF}(T)^c \subset \Sigma(T);\]
					\item \emph{discrete spectrum}  $\Sigma_d(T)$ of $T$ as the set  \[\Sigma_d(T) := \Sigma(T) \setminus \Sigma_{ess}(T).\]
				\end{itemize}
			\end{definition}

			The two main results of this paper are the following Theorems. We first localize the spectrum of the operator $\LL_{\alpha, \gamma}$ in the space $L^1\left (m^{-1}\right )$, generalizing to this space the classical $L^2$ result of Nicolaenko \cite{Nicolaenko:71} (see also the chapter 7 of the book \cite{CIP:94} of Cercignani, Illner and Pulvirenti).	Let us denote by $\Delta_x$ for $x \in \RR$ the half-plane 
	    \[ \Delta_x := \{ \zeta \in \mathbb{C} : \Re e \, \zeta \geq x \}.\]
	    We first prove the following result (which has been summarized in Figure \ref{figSpect}).

		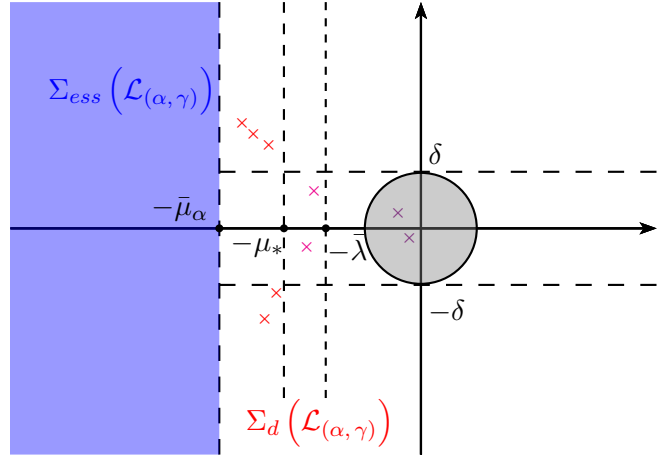
\begin{figure}
			\begin{center} 
				\psset{xunit=0.5cm,yunit=0.5cm,algebraic=true,dotstyle=o,dotsize=3pt 0,linewidth=0.8pt,arrowsize=3pt 2,arrowinset=0.25}
				\begin{pspicture*}(-10,-6)(8.21,6)
          \psaxes[xAxis=true,yAxis=true,labels=none,Dx=2,Dy=2,ticksize=0pt 0]{->}(.8,0)(-10,-6)(7,6)
					\psline[linestyle=dashed,dash=6pt 6pt](-4.5,-6)(-4.5,6)
					\psline[linestyle=dashed,dash=4pt 4pt](-2.8,-4.5)(-2.8,6)
					\psline[linestyle=dashed,dash=3pt 3pt](-1.7,-4.5)(-1.7,6)
    			\psline[linestyle=dashed,dash=6pt 6pt](-4.5,1.5)(7,1.5)
					\psline[linestyle=dashed,dash=6pt 6pt](-4.5,-1.5)(7,-1.5)
					\rput[bl](1.,1.7){$\delta$}
					\rput[bl](1.,-2.5){$-\delta$}
					\rput[bl](-9.,3){\blue{$ \Sigma_{ess} \left (\mathcal L_{(\alpha, \, \gamma)}\right ) $}}
					\pspolygon[linestyle=none,fillstyle=solid,fillcolor=blue,opacity=0.4](-12,-6)(-4.5,-6)(-4.5,6)(-12,6)
					\psdots[dotstyle=*](-4.5,0)
					\rput[br](-4.8,0.2){$-\bar \mu_{\alpha}$}
					\psdots[dotstyle=*](-2.8,0)
					\rput[tr](-2.8,-0.15){$-\mu_*$}
					\psdots[dotstyle=*](-1.7,0)
					\rput[tl](-1.65,-0.15){$-\bar\lambda$}
					\psdots[dotsize=4pt 0, dotstyle=x, linecolor=red](-3.9,2.8)
					\psdots[dotsize=4pt 0, dotstyle=x, linecolor=red](-3.2,2.2)
					\psdots[dotsize=4pt 0, dotstyle=x, linecolor=red](-3.3,-2.4)
					\psdots[dotsize=4pt 0, dotstyle=x, linecolor=red](-3.6,2.5)
					\psdots[dotsize=4pt 0, dotstyle=x, linecolor=red](-3,-1.7)
					\psdots[dotsize=4pt 0, dotstyle=x, linecolor=magenta](-2,1)
					\psdots[dotsize=4pt 0, dotstyle=x, linecolor=magenta](-2.2,-.5)
					\psdots[dotsize=4pt 0, dotstyle=x, linecolor=violet](.5,-.25)
					\psdots[dotsize=4pt 0, dotstyle=x, linecolor=violet](.2,.4)
					\pscircle[fillstyle=solid,fillcolor=gray,opacity=0.4](.8,0){0.75}
					\rput[bl](-3.8,-5.8){\red{$\Sigma_d\left (\mathcal L_{(\alpha, \, \gamma)}\right ) $}}				
				\end{pspicture*}

			\end{center}
			\label{figSpect}
			\caption[Localization of the eigenvalues of $\mathcal{L}_{\alpha, \, \gamma}$ for small frequencies $\gamma$]{Localization of the eigenvalues of $\mathcal{L}_{\alpha, \, \gamma}$ for $|\gamma| \leq \gamma_0(\delta)$. }
		\end{figure}
				
			\begin{theorem}
			  \label{thmLocalizSpect}
	      Let $\alpha \in (\alpha_1, 1]$,  for a constructive constant $0< \alpha_1<1$. There exists a constructive constant $\bar \mu_{\alpha} > 0$ such that the essential spectrum of the operator  $\LL_{(\alpha, \, \gamma)}$ in $W_1^{2,1}\left (m^{-1}\right )$ is contained on the half-plane  $\Delta_{-\bar \mu_{\alpha}}^c$:
	      \[ \Sigma_{ess}\left (\LL_{(\alpha, \, \gamma)}\right ) \subset \Delta_{-\bar \mu_{\alpha}}^c.\]
	      The remaining part of its spectrum is composed of discrete eigenvalues. Their behavior for small frequencies $\gamma$ is the following.
	      
	      Let us fix $\delta > 0$. There exist some constants $0 < \bar \lambda < \mu_* < \mu_\alpha$ and  $\alpha_2 \in(\alpha_1,1]$ such that if $\alpha \in (\alpha_2,1]$ there exists a nonnegative number $\gamma_0$ such that for all $|\gamma| \leq \gamma_0$, if $\lambda \in \Sigma_d(\mathcal{L}_{(\alpha, \, \gamma)})$, then
				\begin{gather*}
					 \lambda \in \Delta_{- \mu_*} \Rightarrow \left | \Im m \, \lambda \right | \leq \delta; \\
					 \lambda \in \Delta_{-\frac{\bar \lambda}{2}} \Rightarrow | \lambda| \leq \delta. 
				\end{gather*} 			  
			\end{theorem}

			We then give a first order (in $\gamma$ and $\alpha$) Taylor expansion of the eigenvalues of  $\LL_{\alpha, \gamma}$, which generalizes the results of Ellis and Pinsky \cite{ellis:1975} and Mischler and Mouhot \cite{mischler:20092}. Notice that this result also contains a part of the analysis led by Brilliantov and P\"oschel in the chapter 25 of their book \cite{brilliantov:2004} about the stable and unstable modes of the fluid approximation of the granular gases equation, namely that the energy eigenvalue is proportional to the inelasticity.
			
			 Before stating the result, let us define the eigenvalue problem we want to deal with: finding a triple $(\lambda, \gamma,h)$ such that
			\begin{equation}
			  \label{pbEigenValueLambdaGammaRho}
			  \left (- i(\gamma \cdot v) + \mathcal{L}_\alpha \right )h = \lambda \, h,
			\end{equation}
		  for $\gamma \in \RR^d$, $\lambda \in \CC$ and $h \in L^1\left (m^{-1}\right )$.

			\begin{theorem}
			  \label{thmTaylorExpSpectrum}
			  There exist $\alpha_*  \in(\alpha_2,1]$, some open sets $U_1 \times U_2\subset\mathbb{R} \times \mathbb{C} $, neighborhood of $(0,0)$, and functions 
			  \begin{equation*}
			    \left \{ \begin{aligned}
			      & \lambda^{(j)}: U_1 \times (\alpha_*,1] \to U_2 & \forall \, j \in \{-1,\ldots,d\}, \\
			      & h^{(j)}: U_1 \times \SSS^{d-1} \times (\alpha_*,1] \to L^1\left (m^{-1}\right ) & \forall \, j \in \{-1,\ldots,d\},
			    \end{aligned} \right.
			  \end{equation*}
			  such that
			  \begin{enumerate}[leftmargin=*]
			    \item The triple $\left (\rho \, \omega,  \,\lambda^{(j)}(\rho, \alpha), \, h^{(j)}(\rho, \omega, \alpha\right )$ is solution to the eigenvalue problem \eqref{pbEigenValueLambdaGammaRho}, for all $\alpha \in (\alpha_*,1]$, $\rho \in U_1$, $\omega \in \SSS^{d-1}$, $j \in \{-1,\ldots,d\}$;
			    \medskip
			    \item The eigenvalue $\lambda^{(j)}$ is analytic on $U_1 \times (\alpha_*,1]$ and verifies
						\begin{equation*}
							\left\{ \begin{aligned} 
								& \lambda^{(j)}(0,1) = 0, && \forall \, j \in \{-1,\ldots,d\}, \\
								& \frac{\partial \lambda^{(j)}}{\partial \rho}(0,1) = j \, i \, \sqrt{\bar T_1 + \frac{2\bar T_1^2}{d}}, && \forall \, j \in \{-1,0,1\}, && \frac{\partial \lambda^{(j)}}{\partial \rho}(0,1) = 0, && \forall \, j \in \{2,\ldots,d\},\\
								& \frac{\partial^2 \lambda^{(j)}}{\partial \rho^2}(0,1) < 0 , && \forall \, j \in \{-1,\ldots,d\}, \\
								& \frac{\partial \lambda^{(0)}}{\partial \alpha}(0,1) = - \frac{3}{\bar T_1}, && &&\frac{\partial \lambda^{(j)}}{\partial \alpha}(0,1) = 0, && \forall \, j \in \{-1,1,\ldots,d\};
							\end{aligned}\right.
						\end{equation*}		
					\medskip	
					\item For $\alpha \in (\alpha_*,1]$, if a triple $(\rho \omega,\lambda, h)$ is solution to the problem \eqref{pbEigenValueLambdaGammaRho} for $(\rho, \lambda) \in U_1 \times U_2$, then necessarily $\lambda = \lambda^{(j)}$ for some $j\in \{-1,\ldots,d\}$;
					\medskip 
					\item For $j \in \{-1,0,1\}$, $\alpha \in (\alpha_*,1]$ and $\omega \in \SSS^{d-1}$, the function $v \mapsto h^{(j)}(\rho\, \omega, \alpha)(v)$ depends only on $|v|$ and $v \cdot \omega$.
			  \end{enumerate}
			\end{theorem}
			
			\begin{remark}
			  We notice in this result that the eigenvalues depend only on $|\gamma|$ and not on $\gamma$ itself. This is due to the rotational invariance of the linearized collision operator.
			  However, this is not the case of the eigenvectors, which can also depend on the angular coordinates.
 			\end{remark}
			
			\begin{remark}
			  As a consequence of this result, we can write for $\left (\rho, \alpha\right ) \in U_1 \times (\alpha_*,1]$
			  \[ \lambda^{(j)}(\rho, \alpha) = i \lambda_1^{(j)} \rho -\lambda_2^{(j)} \rho^2  - e_1^{(j)}\left (1-\alpha\right ) + \mathcal{O}\left ( \rho^2 + (1-\alpha)^2\right ),\]
			  for explicit (see Section \ref{subHighOrder}) constants $\lambda_1^{(j)}\in \RR$, $\lambda_2^{(j)}\in \RR_+$  and $e_1^{(j)}\in \RR_+$.
			  In particular, we obtain that for small space frequencies and small values of the inelasticity, the spectrum of the linear operator remains at the left of the imaginary axis in the complex plane.
			  This phenomenon has at least two  important consequences on the behavior of the solution to the granular gases equation:
			  \begin{itemize}
			    \item The solutions to the linear collision equation \eqref{eqScalFlucLin} are $L^1\left (m^{-1}\right )$ stable;
			    \medskip
			    \item The clustering phenomenon (see \eg \cite{brilliantov:2004}) is not possible for quasi-elastic collisions $\alpha \sim 1$, or for systems with a small typical length scale.
			  \end{itemize}
			\end{remark}
			
		\subsection{Method of Proof and Plan of the Paper}
		  \label{subMethProof}
		  
		  The proof of Theorem \ref{thmLocalizSpect}, concerning the rough\footnote{By rough, we mean more precisely that this result do not establish whether or not the eigenvalues can cross the vertical axis of the complex plane.} localization of the spectrum of the linear operator $\LL_{(\alpha, \, \gamma)}$ is given in Section \ref{secLocaliz}, and can be summarized as follows:
	  	\begin{itemize}
		    \item We first decompose this operator as a sum of compact and Schr\"odinger-like operators:
		      \begin{align*}
	            \mathcal L_{(\alpha, \, \gamma)} \, h & = -\left [i \, (\gamma \cdot v) + \nu_\alpha(v) \right ] h + (1-\alpha) \Delta_v h +  2 \, \Q_\alpha^+\left ( F_\alpha, h \right ) - F_\alpha \, L(h) \\
	             & = D_{(\alpha, \, \gamma)} h + \LL_\alpha^c h.
	          \end{align*}
		    \item We then compute the spectrum in $L^1$ of the Schr\"odinger-like part $ D_{(\alpha, \, \gamma)}$ and apply  a Banach variant of Weyl's Theorem (found \eg in \cite{kato:1966}) about the stability of the essential spectrum under relatively compact perturbation
		      \[ \Sigma_{ess} \left ( \mathcal L_{(\alpha, \, \gamma)}\right ) = \Sigma_{ess} \left (D_{(\alpha, \, \gamma)}\right ). \]
		    \item Once the essential spectrum has been localized, we take advantage of some space homogeneous coercivity and spectral gap estimates (proven in \cite{mischler:20091}) to establish the existence of the eigenvalues.
		    \smallskip
		    \item We finally combine some information about the asymptotic ($\alpha \to 1$) behavior of the space homogeneous eigenvalues (also taken from \cite{mischler:20091}) and the decay properties of the semi-group of the operator $D_{(\alpha, \, \gamma)}$ to finally give a rough localization of the spectrum.
		  \end{itemize}
		  
		  The proof of Theorem \ref{thmTaylorExpSpectrum} concerning the Taylor expansion of the eigenvalues of  $\LL_{(\alpha, \, \gamma)}$ is given in Section \ref{secEigTaylor}. Up to a certain extent. this proof is a generalization of Ellis \& Pinsky's arguments for \cite{ellis:1975}, namely: 
		  \begin{itemize}
		    \item We start by reformulating the eigenvalue problem  \eqref{pbEigenValueLambdaGammaRho} as a functional equation, using bounded operators:
		      \[ \eqref{pbEigenValueLambdaGammaRho} \Longleftrightarrow \text{Finding } (\lambda, \gamma,h) \text{ s.t. } h = \Psi_{(\lambda,\gamma, \, \alpha)}^{-1} \Phi_{(\lambda, \gamma, \, \alpha)} \nu_\alpha^{-1/2} \Pi \,\nu_\alpha^{1/2} h, \]
		      where $\Pi$ is a projection operator, $\Phi$ a ``multiplication'' operator, and $\Psi$ a ``small'' perturbation of the identity.
		    \item We then project this new problem onto the space of \emph{elastic} collisional invariants, allowing to rewrite completely \eqref{pbEigenValueLambdaGammaRho} as a finite dimensional system of linear equations of the form
					\begin{equation*} 
						\left (A_{(\lambda,\gamma, \alpha)} - \id\right ) X_{(\lambda,\gamma, \alpha)} = 0,
					\end{equation*}
					for a non-invertible square matrix $A$.
				\smallskip
		    \item We finally solve this system of equation taking advantage of some elastic and space homogeneous techniques, from both \cite{ellis:1975} and \cite{mischler:20091}.
		  \end{itemize}

  \section{Localization of the Spectrum}
    \label{secLocaliz}
  
    In this section, we shall give a rough localization of the spectrum of the linearized collision operator $\LL_\alpha$, proving Theorem \eqref{thmLocalizSpect}.

    \subsection{Geometry of the Essential Spectrum}
    
      We start by describing the ``easy part'', namely the essential spectrum. As we have to deal with the Banach space $L^1\left (m^{-1}\right )$, we cannot apply directly the classical Weyl's Theorem about the stability of the spectrum under relatively compact perturbations, because of the lack of Hilbertian structure. We shall rather apply the more general version stating only the stability of the semi-Fredholm set, which is well suited for our definition of the essential spectrum.
      
	    \begin{proposition}
	      \label{propLocalisEssSpect}
	      Let $\alpha \in (\alpha_0, 1]$, where $\alpha_0$ is defined in Lemma \ref{lemVracProp}. There exists a constructive constant $\bar \mu_{\alpha} > 0$ such that the essential spectrum of the operator  $\LL_{(\alpha, \, \gamma)}$ in $W_1^{2,1}\left (m^{-1}\right )$ is contained on the half-plane  $\Delta_{-\bar \mu_{\alpha}}^c$:
	      \[ \Sigma_{ess}\left (\LL_{(\alpha, \, \gamma)}\right ) \subset \Delta_{-\bar \mu_{\alpha}}^c.\]
	      The remaining part of its spectrum is composed of discrete eigenvalues.
	    \end{proposition}
	    
	    \begin{proof}
	      Let us use the expression \eqref{eqOpLinFour} of the collision operator in spatial Fourier variables, and decompose it for $h \in L^1\left (m^{-1}\right )$ as a local and a non local part:
	      \begin{align}
	        \LL_{(\alpha, \, \gamma)} \, h & = -\left [i \, (\gamma \cdot v) + \nu_\alpha(v) \right ] h + (1-\alpha) \Delta_v h +  2 \, \Q_\alpha^+\left ( F_\alpha, h \right ) - F_\alpha \, L(h) \notag \\
	         & = D_{(\alpha, \, \gamma)} h + \LL_\alpha^c h, \label{defDecompoLalphgam}
	      \end{align}
	      where 
	      \begin{equation}
	        \label{defDK}
	        \left \{ \begin{aligned}
	          & D_{(\alpha, \, \gamma)} := -\left [i \, (\gamma \cdot v) +  \nu_\alpha(v)\right ] \id + (1-\alpha) \Delta_v, \\
	          & \LL_\alpha^c := 2 \Q_\alpha^+\left ( F_\alpha, \cdot \right ) - F_\alpha \, L(\cdot).
	        \end{aligned} \right.
	      \end{equation}
	      We start by  the spectrum of $D_{(\alpha, \, \gamma)}$ in $L^1\left (m^{-1}\right )$. 
	      This operator is the difference of a Laplace operator with the operator of multiplication by $C_{(\alpha, \, \gamma)}(v) := i \, (\gamma \cdot v) + L(F_\alpha)$. This quantity verifies according to the lower bound of the collision frequency \eqref{nuCrois}
	      \[\Re e \, C_{\alpha, \, \gamma}(v) > \nu_{0,\alpha},\]
	      where $\nu_{0,\alpha}$ is  the lower bound of the loss term $v \to L(F_\alpha)(v)$ (thanks to the smoothness of the profile $F_\alpha$ stated in Proposition \ref{propHoldContOp}).
	      
	      It is known from \eg \cite{HempelVoigt:86,Kunstmann:1999} that the spectrum of the Schr\"odinger-like operator $D_{(\alpha, \, \gamma)}$ is independent of the weighted $L^p$ space (for $p \in [1, +\infty)$) where we study it. Let us compute this spectrum in $L^2$.
	      To this end, we shall look at the stability properties of the semi-group generated by $D_{(\alpha, \, \gamma)}$ on $L^2\left (m^{-1}\right )$: let  $h=h(t,v) \in \mathcal C \left (0, +\infty; W_1^2\right )$ be a weak solution to
	      \begin{equation}
	        \label{eqSchrod}
	        \frac{\partial h}{\partial t} = D_{(\alpha, \, \gamma)} h.
	      \end{equation}
	      If we multiply this equation by $\bar h$ and integrate in the velocity space, we have thanks to Stokes Theorem and for $\alpha$ close to $1$
	      \begin{align}
	        \frac{\partial}{\partial t}  \|h(t)\|_{L^2\left (m^{-1}\right )}^2 \leq  -  \| C_{(\alpha, \, \gamma)} \, h(t)\|_{L^2\left (m^{-1}\right )} \leq - \nu_{0,\alpha}\|h(t)\|_{L^2\left (m^{-1}\right )}^2. \label{ineqDecaySchrod}
	      \end{align}
	      We then obtain that
	      \begin{equation*}
	        \|h(t)\|_{L^2\left (m^{-1}\right )} \leq e^{- \nu_{0,\alpha} \, t /2}.
	      \end{equation*}
	      Hence, there exists a constant $0 <  \bar \mu_{\alpha} < \nu_{0,\alpha} $ such that the spectrum of the operator $D_{(\alpha, \, \gamma)}$ is included in the set $\Delta_{-\bar \mu_{\alpha}}^c$.
	      
	      Moreover, thanks to the H\"older continuity of the inelastic gain term in operator norm (with loss of weight) stated in Proposition \ref{propHoldContOp}, the operator $\LL_\alpha^c$ is $D_{(\alpha, \, \gamma)}$--compact (and this is here that the weak inelasticity assumption $\alpha \in (\alpha_0, 1]$ is used). 
	       Notice that we have chosen to define the essential spectrum of an operator $S$ ``à la Kato'', namely as the complement of the semi-Fredholm set of $S$ in $\CC$. Then we can apply the Banach version of Weyl's Theorem (see \eg \cite[Theorem  IV.5.26 and IV.5.35]{kato:1966}),  stating that the semi-Fredholm set is stable under relatively compact perturbation. Hence, the essential spectrum of $\LL_{(\alpha, \, \gamma)}$ is included in $\Delta_{-\bar \mu_{\alpha}}^c$.
	       
	       The set $\Delta_{-\bar \mu_{\alpha}}$ is then equal to the Fredholm set $\mathcal{F}\left (\LL_{(\alpha, \, \gamma)}\right )$.
	       It remains to show that this set only contains the eigenvalues and the resolvent set. We know from the discussion in \cite[Chapter IV, Section 6, and Theorem 5.33]{kato:1966} that $\mathcal{F}\left (\LL_{(\alpha, \, \gamma)}\right )$ is an open set, composed of the union of a countable number of components $ \mathcal F_n$, characterized by the value of the index: for any $n \in \NN$, the functions 
	       \[\nullity :\zeta \to \nullity\left (\LL_{(\alpha, \, \gamma)} - \zeta\right ), \qquad \deficiency:\zeta \to \deficiency\left (\LL_{(\alpha, \, \gamma)} - \zeta\right )\]
	       are constant on $\mathcal F_n$, except for a countable set of isolated values of $\zeta$.
	       In our case, we have $\mathcal{F}\left (\LL_{(\alpha, \, \gamma)}\right ) = \Delta_{-\bar \mu_{\alpha}}$ which is connected; it has only one component, which means that $\nullity(\zeta)$ and $\deficiency(\zeta)$ are constant on $\Delta_{-\bar \mu_{\alpha}}$,  except for a countable set of isolated values of $\zeta$.
	       
	       We will prove that these constant values are $\nullity(\zeta) = \deficiency(\zeta) = 0$, meaning that $\zeta$ belongs to the resolvent set of $\LL_{(\alpha, \, \gamma)}$. 
	       The remaining isolated values $\zeta$, being in the Fredholm set, then verify $0 < \nullity(\zeta) < +\infty$ and $0 < \deficiency(\zeta) < +\infty$, which exactly characterizes the eigenvalues. We shall follow closely the proof of \cite[Proposition 3.4]{Mouhot:2006}, and exhibit an uncountable set $I \subset \Delta_{-\bar \mu_{\alpha}}$ such that $\nullity(\zeta) = \deficiency(\zeta) = 0$ for all $\zeta \in I$.
	       
	       Let us use the decomposition \eqref{defABops} introduced initially in \cite{mischler:20091}
	       \begin{equation*}
	         \LL_{\alpha, \, \gamma} = A_\delta - B_{\alpha, \, \delta}\left (i (\gamma \cdot v)\right ),
	       \end{equation*}
	       for $\delta > 0$ (see Section \ref{secFuncToolbox} for more details). We know from Lemma \ref{lemVracProp} that $A_\delta$ is compact on $L^1\left (m^{-1}\right )$, and that $B_{\alpha,\delta}$ satisfies the coercivity estimate
	       \begin{equation}
	         \label{ineqCoercivB}
	         \left \| B_{\alpha, \, \delta}(\zeta) \,g \right \|_{L^1\left (m^{-1}\right )} \geq \left \| \left (\nu_1 + \Re e \, \zeta\right ) g \right \|_{L^1\left (m^{-1}\right )} - \ve(\delta)  \left \| \left (\nu_1 + \Re e \, \zeta\right ) g \right \|_{L^1\left (m^{-1}\right )},
	       \end{equation}
	       where $\ve(\delta) \to 0$ where $\delta \to 0$. If  we fix $r_0 > 0$ sufficiently big and  $\delta>0$ small enough, then we have according to \eqref{ineqCoercivB} for all $r \geq r_0$
	       \begin{equation*}
	         \left \| B_{\alpha, \, \delta}\left (r + i (\gamma \cdot v) \right ) g \right \|_{L^1\left (m^{-1}\right )} \geq \frac{\nu_{0,1} + r_0}{2} \| g \| _{L^1\left (m^{-1}\right )}.
	       \end{equation*}
	       Thus, the operator $B_{\alpha, \, \delta}\left (r + i (\gamma \cdot v) \right )$ is invertible on $L^1\left (m^{-1}\right )$, for all $r \geq r_0$, and then it is the same for $ \LL_{\alpha, \, \gamma} - r = A_\delta - B_{\alpha, \, \delta} \left (r + i (\gamma \cdot v) \right )$ by compacity of $A_\delta$. It finally means that the interval $I:=[r_0, +\infty)$ is included on the resolvent set of $\LL_{\alpha, \, \gamma}$, and then that
	       \[ \nullity(\zeta) = \deficiency(\zeta) = 0, \quad \forall \, \zeta \in [r_0, +\infty), \]
	       which concludes the proof.
	    \end{proof}

		\subsection{Behavior of the Eigenvalues for Small Inelasticity}	

      We shall now focus on the discrete spectrum of this operator, namely its eigenvalues. A major difference with the elastic case in the classical Hilbertian $L^2$ setting is that the operator we deal with is not a nonpositive operator, and we cannot conclude thanks to the last proposition that this operator has a \emph{spectral gap} (namely a negative bound for its eigenvalues).
      Nevertheless, we know from \cite{Mouhot:2006} for the elastic case and \cite{mischler:20092} for the weak inelasticity case $\alpha \in (\alpha_0, 1)$ that $\LL_\alpha$ has a spectral gap $- \bar \lambda$ in $L^1\left (m^{-1}\right )$, verifying for $\alpha$ sufficiently small (say $\alpha \in (\alpha_1, 1]$ for $1 > \alpha_1 > \alpha_0$)
      \[ 0 <  \bar \lambda < \mu_* < \bar \mu_{\alpha},\]
      for a nonnegative constant $\mu_*$ depending on $\alpha$.

		Let us now study the behavior of the discrete spectrum of  $\mathcal{L}_{(\alpha, \, \gamma)}$ for small values of the frequency $\gamma$. We shall show that if $\gamma \to 0$, then the eigenvalues of this operator converge first towards the real axis and then towards $0$.
		
		\begin{proposition} 
		  \label{propLocalisDiscSpect}
			Let $\delta > 0$. There exists $\alpha_2 \in(\alpha_1,1]$ such that if $\alpha \in (\alpha_2,1]$ there exists a nonnegative number $\gamma_0$ such that for all $|\gamma| \leq \gamma_0$, if $\lambda \in \Sigma_d(\mathcal{L}_{(\alpha, \, \gamma)})$, then
			\begin{gather}
				 \lambda \in \Delta_{- \mu_*} \Rightarrow \left | \Im m \, \lambda \right | \leq \delta; \label{imlambda}\\
				 \lambda \in \Delta_{-\frac{\bar \lambda}{2}} \Rightarrow | \lambda| \leq \delta. \label{valambda}
			\end{gather} 
		\end{proposition}
		
		\begin{proof}
		  Let us first notice that if $\lambda$ is an eigenvalue of $\mathcal{L}_{(\alpha, \, \gamma)}$ and $h$ an associated eigenvector, then using the decomposition \eqref{defDecompoLalphgam} introduced in the proof of Proposition \ref{propLocalisEssSpect}, we can write
		  \begin{equation}
		    \label{eqDecompoEigVal}
		    \LL_\alpha^c h = \left ( \lambda - D_{(\alpha, \, \gamma)} \right ) h,
		  \end{equation}
		  where $\LL_\alpha^c$ is compact on $L^1\left (m^{-1}\right )$ (thanks to the sharp estimates of Lemma \ref{lemVracProp}) and 
		  \[D_{(\alpha, \, \gamma)} = -\left [i \, (\gamma \cdot v) + \nu_\alpha(v) \right ] \id + (1-\alpha) \Delta_v .\]
		  
		  We will proceed by contradiction using the representation \eqref{eqDecompoEigVal}. Concerning the first implication, if, for $\delta > 0$, there exist a sequence $(\gamma_n)_n \subset \mathbb{R}^d$ converging towards $0$, a sequence of functions $(h_n)_n \in L^1\left (m^{-1}\right )$ of unit norm, and a sequence of complex numbers $\lambda_n \in \Sigma_d\left (\mathcal{L}_{(\alpha,\,\gamma_n)}\right )$ verifying 
			\begin{equation} 
			  \label{eqDecompoEigValn}
			  \left \{ \begin{aligned}
			    & \LL_\alpha^c h_n = \left ( \lambda _n- D_{(\alpha, \, \gamma_n)} \right ) h_n, \\
				  & \left | \Im m \, \lambda_n \right | > \delta, \quad \Re e \,\lambda_n \geq - \mu_*,
				\end{aligned} \right.
			\end{equation}
			then we must have $\limsup \left| \Im m \, \lambda_n \right | < \infty$. Indeed, the operator $\LL_\alpha^c$ is compact on $L^1\left (m^{-1}\right )$, and then the sequence $(\LL_\alpha^c h_n)_n$ converges (up to an extraction) towards $g \in L^1\left (m^{-1}\right )$. Thus we can write using \eqref{eqDecompoEigValn}
			\begin{equation}
			  \label{eqLimGKFn}
				g =  \lim_{n \to \infty} \left ( \lambda_n - D_{(\alpha, \, \gamma_n)} \right ) h_n.
			\end{equation}
			
			We have seen in the proof of Proposition \ref{propLocalisEssSpect} that the semi-group $S_t^{(\alpha, \, \gamma)}$ associated to the operator $D_{(\alpha, \, \gamma)}$ in $L^1$ is exponentially decaying in time, uniformly in $\alpha$ and $\gamma$.
			But, we know (see \eg \cite{engel:2000}, chap. II) that if $\mathcal R(D_{(\alpha, \, \gamma)}, \cdot)$ is the resolvent operator of $D_{\alpha, \, \gamma}$, we have the integral representation for all $\zeta \in R(D_{(\alpha, \, \gamma)})$
			\begin{equation*}
			  \mathcal R\left (D_{(\alpha, \, \gamma)}, \zeta\right ) = \lim_{t \to +\infty} \int_0^{t} e^{-t \,\zeta} S_t^{(\alpha, \, \gamma)} \, dt. 
			\end{equation*}
			Thus, using the decay of $S_t^{(\alpha, \, \gamma)}$, we have $ \mathcal R\left (D_{(\alpha, \, \gamma)}, 0\right ) =: D_{(\alpha, \, \gamma)}^{-1} $ bounded in $L^1$, uniformly in  $\gamma$, and then according to \eqref{eqLimGKFn}
			\begin{equation}
			  \label{eqLimFng} 
			  \lim_{n \to \infty} h_n = \left ( \lim_{n \to \infty} \lambda_n - D_{(\alpha, \, 0)}\right )^{-1} g. 
			\end{equation}
			But, we also have for $v \in \RR^d$
			\begin{equation*}
			  \left ( \lambda_n - D_{(\alpha, \, \gamma_n)}\right )^{-1} g(v) = \frac{1}{\lambda_n + i ( \gamma_n \cdot v) + \nu_\alpha(v) } \left ( \id - \frac{1- \alpha}{\lambda_n + i ( \gamma_n \cdot v) + \nu_\alpha(v) } \Delta_v  \right )^{-1} g(v),
			\end{equation*}
			and then by considering again the behavior of the solutions to equation \eqref{eqSchrod}, which gives inequality \eqref{ineqDecaySchrod}, we obtain a constant $C$ independent on $(\alpha, \lambda_n, \gamma_n)$ such that
			\begin{equation}
			  \label{ineqLambPlusDGLinf}
			  \left \| \left ( \lambda_n - D_{(\alpha, \, \gamma_n)}\right )^{-1} g \right \|_{L^\infty} \leq  \frac{C}{|\lambda_n| - \nu_{0,\alpha}} \|g\|_{L^\infty}.
			\end{equation}
			Finally, if $\lim \left| \Im m \, \lambda_n \right | = \infty$ we would have according to the limit \eqref{eqLimFng} and the estimation \eqref{ineqLambPlusDGLinf} 
			\[\lim_{n \to \infty} \| h_n \|_{L^\infty} = 0\]
			with $\|h_n\|_{L^1\left (m^{-1}\right )} = 1$, which is not possible.
			Hence, $\left| \Im m \, \lambda_n \right| \leq C$ for an infinite number of indices $n$ and $C >0$. 
			
			But, we also have $-\mu_* \leq \Re e \,\lambda_n < r_0 $ (where $r_0 > 0$ is defined in the proof of Proposition \ref{propLocalisEssSpect}), and then we can extract another subsequence $\left (\lambda_{n_k}\right )_k$ converging towards $\lambda \in \CC$ such that $\Im m \, \lambda \geq \delta > 0$. Using the fact that $\gamma_n \to 0$ and the smoothness of the map $\lambda \mapsto \left (\lambda - D_{(\alpha, \, 0)}\right )^{-1}$, we obtain in \eqref{eqLimFng}
			\begin{equation*}
				\lim_{k \to \infty} h_{n_k} = \left (\lambda - D_{(\alpha, \, 0)}\right )^{-1} g =: h \in L^1\left (m^{-1}\right ),
			\end{equation*}
			with $\|h\|_{L^1\left (m^{-1}\right )} = 1$.
			Hence we conclude by inversion of $ \left (\lambda - D_{(\alpha, \, 0)}\right )^{-1}$ and by the smoothness of the nonlocal part of $\LL_\alpha$ that 
			\begin{align*}
			  \left (\lambda - D_{(\alpha, \, 0)}\right ) \, h = g = \lim_{n \to \infty} \LL_\alpha^c h_n = \LL_\alpha^c h 
			\end{align*} 
			which means according to the definition of $\LL_\alpha$ that 
			\[\lambda h = \mathcal{L}_\alpha h. \]
			This is absurd because $\left | \Im m \, \lambda \right |\geq \delta > 0$, $\Re e \,\lambda \geq - \mu_*$ for $\mu_*$ close to the spectral gap of $\LL_\alpha$, and  yet we know from \cite{mischler:20091} that the eigenvalues of $\LL_\alpha$ can be made arbitrarily close (with respect to $1-\alpha$) to the ones of $\LL_1$, which are real according to \cite{Mouhot:2006}.
			
			We shall now give the proof of the implication \eqref{valambda}, also by contradiction. 
			If for $\delta > 0$ there exist a sequence $(\gamma_n)_n$ converging towards $0$, a sequence $(h_n)_n \in L^1\left (m^{-1}\right )$ of unit norm, and some complex numbers $\lambda_n \in \Sigma_d\left (\mathcal{L}_{\ve, \gamma_n}\right )$ such that
			\begin{equation*}
				- \bar\lambda/2 \leq \Re e \,\lambda_n \leq -\delta,
			\end{equation*}
			then $ \lambda_n \in \Delta_{- \mu_*}$ and according to the relation \eqref{imlambda} we have $\left | \Im m \, \lambda_n \right | \leq \delta$. We can then extract a subsequence $\left (\lambda_{n_k}\right )_k$ which converges towards a complex number $\lambda$ also verifying $- \bar\lambda/2 \leq \Re e \,\lambda \leq -\delta$. 
			When $k \to \infty$, the same argument than before gives $\mathcal{L}_\alpha h = \lambda h$ with $\lambda \neq 0$.  By using again the spectral properties of $\LL_\alpha$, we then have $\Re e \, \lambda \leq -\bar\lambda$, which is absurd.
		\end{proof}
		
		This concludes the proof of Theorem \ref{thmLocalizSpect}. We also summarized the results of this proposition in Figure \ref{figSpect}.
		Moreover, it gives us some rough information on the behavior of the resolvent operator of $\mathcal{L}_{(\alpha, \gamma)}$.
		\begin{corollary}
			 If $\alpha \in (\alpha_1,1]$, the resolvent operator $\mathcal{R}(\mathcal{L}_{(\alpha, \, \gamma}), \zeta)$ is well defined for $\zeta \in \Delta_{-\mu_*}$ such that $\Im m \, \zeta > \delta$.
		\end{corollary}

	\section{Inelastic Dispersion Relations}
		\label{secEigTaylor}
		
		Our goal in this section is to precise the localization results of the previous section, by proving Theorem \ref{thmTaylorExpSpectrum}, that is to give a Taylor expansion of the eigenvalues of $\LL_{(\alpha, \gamma)}$ in $\alpha$ and $\gamma$.
		The purpose of this expansion is twofold: on the one hand, we want to establish that, at least for small values of $\alpha$ and $\gamma$, the eigenvalues of the linear operator $\LL_{(\alpha, \gamma)}$ stay at the left of the imaginary axis. This could be useful if \eg one wants to prove nonlinear stability of the solutions to \eqref{eqBoltzEPS}.
		On the other hand, obtaining this decomposition up to the second order in the spatial frequency $\gamma$  and to first order in $\alpha$ is necessary to establish the validity of the linearized quasi-elastic hydrodynamic limit of our model, in the same way than \cite{ellis:1975}.
		To obtain this expansion, we shall refine the method of proof of this paper, together with the use of some ideas introduced in \cite{mischler:20092} in order to deal with the quasi-elastic setting.
		
		We recall for the reader's convenience that we are interested in the following eigenvalue problem: finding $\lambda\in \CC$, $\gamma \in \RR^d$ and $h \in L^1(m^{-1})$ such that 
		\begin{equation*}
			\left (- i(\gamma \cdot v) + \mathcal{L}_\alpha \right )h = \lambda \, h,
		\end{equation*}
		which can be reformulated thanks to the decomposition \eqref{defDecompoLalphgam} as finding $\lambda\in \CC$, $\gamma \in \RR^d$ and $h \in L^1(m^{-1})$ such that 
		\begin{equation} 
		  \label{pbEigen}
			\mathcal L_\alpha^c h = \left (\lambda + \nu_\alpha(v) + i(\gamma \cdot v) - (1-\alpha) \Delta_v \right) h.
		\end{equation}

	\subsection{Projection of the Eigenvalue Problem} 
	  \label{subProjEig}
		
		Let us now define the scalar product we will use in the following. If $\phi, \psi$ are such that the following expression has a meaning, we will set
		\begin{equation*}
			\langle \phi, \psi \rangle := \int_{\RR^d} \phi(v) \, \overline{\psi}(v) \, F_1^{-1}(v) \,dv,
		\end{equation*}
	  where $F_1 = \mathcal M_{1,0,\bar T_1}$ is the quasi-elastic equilibrium and $\bar T_1$ is given by \eqref{defElasticEquilibTemperature}.
	  Indeed, our goal is to introduce a \emph{spectral decomposition} of $L^1\left (m^{-1}\right )$ as a direct sum of $\LL_\alpha$-invariant spaces. 
	  The inner product we use for this purpose is the one of $L^2\left (F_1^{-1} \right ) $ because $f = m + h \in L^1\left (m^{-1}\right )$ if and only if $h \in L^1\left (m^{-1}\right )$ and $f \in L^2\left (F_1^{-1}\right )$ if and only if $h \in L^2\left (F_1^{-1}\right )$.
	  This hilbertian structure allows us to define the spectral projections.
			
		Let us start by decomposing the operator $\mathcal L_\alpha^c$ as
		\begin{equation} 
		  \label{defK}
			\mathcal L_\alpha^c = \nu_\alpha^{1/2} \left ( \Pi + \mathcal S_\alpha \right ) \nu_\alpha^{1/2},
		\end{equation}
		where $ \Pi$ is the projection on the space 
		\[\mathcal{N}_\alpha := \nu_\alpha^{1/2} N_1 = \vect \{ \nu_\alpha^{1/2} F_1, \, \nu_\alpha^{1/2} v_i \, F_1 , \, \nu_\alpha^{1/2} |v|^2 \, F_1 : \, 1 \leq i \leq d \},\]
		and $\mathcal S_\alpha$ is given by
		\[ \mathcal S_\alpha := \nu_\alpha^{-1} \LL_\alpha^{c} - \nu_\alpha^{-1/2} \Pi \, \nu_\alpha^{1/2}.\]
		Actually, $\mathbb P := \nu_\alpha^{- 1/2} \Pi $ is the spectral projection on the null space of $\LL_1$, and can be defined using the resolvent operator of $\LL_1$ as
		\begin{align*} 
		  \mathbb P  & = \frac{1}{2 i \pi} \int_{\zeta \in \CC : |\zeta| = r} \mathcal{R}(\LL_1, \zeta) \, d \zeta
		\end{align*}
		where $r < \delta$ for $\delta$ sufficiently small (see the discussion in Section $5$ of \cite{mischler:20091}).
		In particular, $\mathbb{P}$ commutes with $\LL_1$.
		Moreover, the operator $\mathcal L_\alpha^c$ being compact on $L^1(m^{-1})$, $\Pi + \mathcal S_\alpha$ is compact on the same space.	Given that the rank of $\Pi$ is finite, $\mathcal S_\alpha$ is then a compact operator on $L^1(m^{-1})$.
		
		We first need to prove a result concerning the eigenvalues of the operator $\nu_1^{-1/2} \mathcal S_1 \, \nu_1^{1/2}$.
		
		\begin{lemma} 
		  \label{Lemvp1}
			On the space $L^1(m^{-1})$, $\lambda = 1$ is not an eigenvalue of $\nu_1^{-1/2} \mathcal S_1 \, \nu_1^{1/2}$.
		\end{lemma}

		\begin{proof}
			If there is an $h\in L^1(m^{-1})$ such that $\nu_1^{-1/2} \mathcal S_1 \, \nu_1^{1/2} h = h$, then according to \eqref{defK},
			\begin{align}
				\LL_1 \, h & = - \nu_1 h + \nu_1^{1/2} ( \Pi + \mathcal S_1 ) \, \nu_1^{1/2} h \notag \\
				 & = \nu_1^{1/2} \,\Pi \, \nu_1^{1/2} h. \label{Pimuf}
			\end{align}
			Projecting this equation upon the null space of $\LL_1$, using \eqref{Pimuf}, we obtain for all $\varphi \in N_1$  that
			\begin{align*}
			  0 & = \langle \LL_1 h, \varphi \rangle \\
			    & = \langle \Pi \, \nu_1^{1/2} h, \nu_1^{1/2} \, \varphi \rangle,
			\end{align*}
			and then $\Pi \, \nu_1^{1/2} h$ is orthogonal to $ \nu_1^{1/2} N_1$, which is the whole range of $\Pi$. Then, necessarily, as $\Pi$ is a projection, we have $\Pi \, \nu_1^{1/2} h = 0$. It means according to \eqref{Pimuf} that $\LL_1 \, h =0$, and then that $h \in N_1$, which is absurd because $\Pi$ is a projection onto $\nu_1^{1/2} N_1$ thus
			\[    \nu_1^{1/2} h = \Pi \, \nu_1^{1/2} h = 0. \]

		\end{proof}

		Let us now denote by $\Phi_{(\lambda, \gamma, \alpha)}$ the operator
		\begin{equation*}
			\Phi_{(\lambda, \gamma, \alpha)} = \left ( \nu_\alpha(v) + \lambda + i \, (\gamma \cdot v)  - (1-\alpha) \Delta_v\right )^{-1} \nu_\alpha(v).
		\end{equation*}
		If the triple $(\gamma, \lambda, h)$ is solution to the eigenvalue problem \eqref{pbEigen}, then we can write using the decomposition \eqref{defK}
		\begin{align}
			h & = \left ( \nu_\alpha(v) + \lambda + i \, (\gamma \cdot v)  - (1-\alpha) \Delta_v\right )^{-1} \nu_\alpha(v) \, (\nu_\alpha(v))^{-1} \LL_\alpha^c h \notag \\
			 & = \Phi_{(\lambda, \gamma, \alpha)} \nu_\alpha^{-1/2} \left ( \Pi + \mathcal S_\alpha \right) \nu_\alpha^{1/2} h. \label{eqfPhi}
		\end{align}
		This will allow us to rewrite the eigenvalues problem with bounded operators. 		
		To this purpose, we state a technical lemma about the asymptotic behavior of the operator $\Phi_{(\lambda, \gamma, \alpha)}$.
		
		\begin{lemma}
		  \label{lemAsymptBehavPhi}
		  For all $g \in W^{2,1}(m^{-1})$, we have
		  \begin{equation*}
		    \left \| \left ( \Phi_{(\lambda, \gamma, \alpha)} - \id \right ) g \right \|_{L^1(m^{-1})} \leq \ve(\lambda, \gamma, \alpha) \|g \|_{W^{2,1}(m^{-1})} ,
      \end{equation*}		
      where $\lim_{(\lambda, \gamma, \alpha) \to (0,0,1)} \ve(\lambda, \gamma, \alpha) = 0$.
		\end{lemma}
		
		\begin{proof}
		  
		  If we set $C_{(\lambda,\, \alpha, \, \gamma)} = \lambda + i \, (\gamma \cdot v) +\nu_\alpha$,  we can write for all $v \in \RR^d$
		  \begin{align*}
		    \Phi_{(\lambda, \gamma, \alpha)} g (v) - g (v) & = \left ( C_{(\lambda,\, \alpha, \, \gamma)}(v)  - (1-\alpha) \Delta_v\right )^{-1} \nu_\alpha(v) g (v) - g (v) \\
		     & = \frac{1}{C_{(\lambda,\, \alpha, \, \gamma)}} \left ( \id - \frac{1-\alpha}{C_{(\lambda,\, \alpha, \, \gamma)}(v)} \Delta_v \right )^{-1}  C_{(0,0,1)} g (v) - g (v).
		  \end{align*}
		  Then, to prove the Lemma, we need to prove that the norm of the operator 
		  \[\mathcal T_\ve := \left ( \id - \ve \, \Delta_v \right )^{-1} - \id,\]
		  defined from $W^{2,1}$ onto $L^1$, can be made arbitrarily small for small $\ve$. 
		  
		  Let us first reformulate this operator using the resolvent of the Laplace operator $\Delta_v$. Since we have
		  \[ \left (\id - \ve \Delta_v\right ) \mathcal{T}_\ve = \ve \Delta_v,\]
		  we can rewrite $\mathcal{T}_\ve$ as\footnote{If it has been possible to conduct this study on a unweighted $L^2$ space, it would have been enough to notice the following Fourier representation: 
		  \[\mathcal{F}_v\left (\mathcal T_\ve \,f \right )(\xi) = -\frac{\ve \, |\xi|^2}{1 + \ve \, |\xi|^2} \mathcal{F}_v(f)(\xi).\]}
		  \begin{align}
		    \mathcal{T}_\ve & =\left ( \frac{1}{\ve}\id - \Delta_v\right )^{-1} \Delta_v =: \mathcal{R}\left ( \frac1\ve , \Delta_v\right ) \Delta_v \label{eqReprResolvTve}.
		  \end{align}
		  It is known from \cite{CholewaDlotko:2004} (and classical for unweighted $L^p$ spaces) that the resolvent of the Laplace operator on exponentially weighted $L^1$ spaces verifies for an explicit real constant $a$ and a nonnegative constant $C$
		  \begin{equation}
		    \label{boundResolvLaplace}
		    \left \|\mathcal{R}\left ( \lambda , \Delta_v\right ) h\right \|_{L^1(m^{-1})} \leq \frac{C}{\lambda -a} \left \| h\right \|_{L^1(m^{-1})},
		  \end{equation}
		  for any $\lambda$ in an unbounded angular sector of the complex plane which does not contains $a$. 
		  In particular, using this inequality in the identity \eqref{eqReprResolvTve}, we have that for any $g \in W^{2,1}(m^{-1})$
		  \[ \left \| \mathcal{T}_\ve \, g\right \|_{L^1(m^{-1})} \leq \frac{\ve \, C}{1 - a \,\ve} \left \| g\right \|_{W^{2,1}(m^{-1})} \to_{\ve \to 0} 0,\]
		  which concludes the proof of the Lemma.

		\end{proof}
		
    In all the following of this section, we shall use the polar decomposition $\gamma = \rho \, \omega$ for $\rho \geq 0$ and $\omega \in \SSS^{d-1}$ of the frequency $\gamma$. 
    We prove an invertibility result, which is needed to rewrite the eigenvalue problem \eqref{pbEigen} using bounded operators:

		\begin{lemma} 
		  \label{lemInvPsi}
			There exist $\alpha_3 \in (\alpha_2, 1]$ and 
			some open sets $U_1 \times U_2 \subset\mathbb{R} \times \mathbb{C} $, neighborhood of $(0,0)$ such that if $(\rho, \lambda,\alpha) \in U_1 \times U_2 \times(\alpha_3,1]$, then for all $\omega \in \mathbb{S}^{d-1}$, the operator 
			\begin{equation*}
				\Psi_{(\lambda, \, \rho\, \omega,\, \alpha)} := \id \, - \, \Phi_{(\lambda, \, \rho \, \omega, \alpha)} \nu_\alpha^{-1/2} \mathcal S_\alpha \, \nu_\alpha^{1/2}
			\end{equation*}
			has a bounded inverse on $ L^1(m^{-1})$.
		\end{lemma}
		
		\begin{proof}
		
			We have seen that $\nu_\alpha^{1/2} \mathcal S_\alpha \, \nu_\alpha^{1/2}$ is a compact operator on $L^1(m^{-1})$. 
			Moreover, we have according to the representation \eqref{eqDecompoEigVal} (with $D_{(\alpha, \, \gamma)}$ given by \eqref{defDK})
			\begin{align*}
			  \Phi_{(\lambda, \, \rho \, \omega, \, \alpha)} & = \left ( \nu_\alpha(v) + \lambda + i \, (\gamma \cdot v)  - (1-\alpha) \Delta_v\right )^{-1} \nu_\alpha(v) \\
			   & =  \left ( \lambda - D_{(\alpha, \, \gamma)} \right )^{-1} \nu_\alpha(v)
			\end{align*}
			which is a bounded operator (at least for small frequencies $\gamma$), thanks to the analysis we led on Section \ref{secLocaliz}, and particularly from the localization of the discrete spectrum of $\LL_{(\alpha, \, \gamma)}$ of Proposition \ref{propLocalisDiscSpect}.
			Hence, the operator
			\[ \Phi_{(\lambda, \, \rho \, \omega, \, \alpha)} \nu_\alpha^{-1/2} \mathcal S_\alpha \, \nu_\alpha^{1/2} = \left ( \Phi_{(\lambda, \, \rho \, \omega, \alpha)} \nu_\alpha^{-1} \right ) \left ( \nu_\alpha^{1/2} \mathcal S_\alpha \, \nu_\alpha^{1/2} \right )\] 
			is compact on $L^1(m^{-1})$. 
			
			By Fredholm alternative, it just remains to show that for  $(\rho, \lambda)$ small enough and $\alpha$ close to $1$, there is no non-trivial solutions $h \in L^1(m^{-1})$ of
			\begin{equation}
			  \label{eqFredholm}
				h =  \Phi_{(\lambda, \, \rho \, \omega, \, \alpha)}\nu_\alpha^{-1/2} \mathcal S_\alpha \, \nu_\alpha^{1/2} h.
			\end{equation}
			Let us do this by contradiction.  
			Assume that there are sequences $(\lambda_n, \, \gamma_n, \, \alpha_n)_n \to (0,0,1)$ and $(h_n)_n \subset L^1(m^{-1})$, $\|h_n\| = 1$ solutions to \eqref{eqFredholm}. 
			First, we notice thanks to the continuity of the equilibrium profiles $F_\alpha$ with respect to $\alpha$ (recalled in Proposition \ref{propCvMaxwProfil}) and to the smoothness properties of these profiles (recalled in Proposition \ref{propRegSelfSimProf}) that we have 
			\[ \lim_{\alpha \to 1} \|\nu_{\alpha} - \nu_1 \|_{L^\infty} = 0.\] 
			Moreover, according to Lemma \ref{lemVracProp}, the operator $\LL_\alpha$ converges towards $\LL_1$ in the norm of graph in $L^1(m^{-1})$. Then, the operator $\nu_{\alpha_n}^{-1/2} \mathcal S_{\alpha_n} \, \nu_{\alpha_n}^{1/2}$, which is compact on $L^1(m^{-1})$, converges towards the operator $\nu_1^{-1/2} \mathcal S_1 \, \nu_1^{1/2} $, and we have up to a subsequence,
			\begin{equation*}
				\nu_{\alpha_n}^{-1/2} \mathcal S_{\alpha_n} \, \nu_{\alpha_n}^{1/2} h_n \to_{n \to \infty} g \in L^1(m^{-1}).
			\end{equation*}
			Thus, if we write $\Phi_n := \Phi_{(\lambda_n, \, \gamma_n, \, \alpha_n)}$, given that $h_n$ is solution to \eqref{eqFredholm}, we have 
			\begin{equation} 
			  \label{eqfnPhin}
				h_n = \Phi_n \left[ \nu_{\alpha_n}^{-1/2}  \mathcal S_{\alpha_n} \, \nu_{\alpha_n}^{1/2} h_n - g \right] + \Phi_n \, g.
			\end{equation}
			But, according to Lemma \ref{lemAsymptBehavPhi} we have
			\begin{equation}
			  \label{eqLimPhingG}
			  \lim_{n \to \infty } \| \Phi_n g - g \|_{L^1(m^{-1})} = 0.
			\end{equation}
			Hence, by Lebesgue dominated convergence Theorem, this implies with identity \eqref{eqfnPhin} that the sequence $(h_n)_n$ strongly converges towards $g$ (using also the fact that the sequence $(\Phi_n g)_n$  is bounded in $L^1(m^{-1})$ for large $n$ thanks to \eqref{eqLimPhingG}).
			Therefore, we have by continuity
			\[g = \lim_{n \to \infty} \nu_{\alpha_n}^{-1/2} \mathcal S_{\alpha_n} \, \nu_{\alpha_n}^{1/2} h_n = \nu_1^{-1/2} \mathcal S_1 \, \nu_1^{1/2} g,\]
			namely $g$ is an eigenvector%
			\footnote{This justifies the use of Lemma \ref{lemAsymptBehavPhi} in equation \eqref{eqLimPhingG}. Indeed, the solutions $h$ to this eigenvalue problem are such that 
		  \[ \| h \|_{W^{2,1}(m^{-1})} \le C \, \| h \|_{L^1(m^{-1})}, \]
		  for $C \geq 0$, as was shown for instance in \cite{mischler:20091} using the decomposition \eqref{defABops}: 
		  \[0 = (\LL_{\alpha, \gamma } - \lambda) h = \left (A_\delta - B_{\alpha, \, \delta} \left (\lambda + i \left (\gamma \cdot v \right )\right )\right )h.\]} 
		  for $\nu_1^{-1/2} \mathcal S_1 \, \nu_1^{1/2}$ associated to the eigenvalue $1$, which is absurd according to Lemma \ref{Lemvp1}. 
			Finally, the operator $\id \, - \, \Phi_{(\lambda, \,\gamma, \, \alpha)} \nu_\alpha^{-1/2} S \,\nu_\alpha^{1/2}$ is invertible on $L^1(m^{-1})$ for small $(\lambda, \, \gamma)$ and $\alpha$ close to $1$.

		\end{proof}

		Let us now rewrite the relation \eqref{eqfPhi} as
		\begin{align*}
			\Psi_{(\lambda,\gamma, \, \alpha)} h & = \left[ \id - \Phi_{(\lambda, \gamma, \, \alpha)} \nu_\alpha^{-1/2} \mathcal S_\alpha \, \nu_\alpha^{1/2} \right]  \Phi_{(\lambda, \gamma)} \nu_\alpha^{-1/2} \left( \Pi + \mathcal S_\alpha \right) \nu_\alpha^{1/2} h \\
			 & = \Phi_{(\lambda, \gamma, \, \alpha)}  \left \{ \nu_\alpha^{-1/2} \Pi \,\nu_\alpha^{1/2} + \nu_\alpha^{-1/2} \mathcal S_\alpha \, \nu_\alpha^{1/2} \left[ \id - \Phi_{(\lambda, \gamma)} \left ( \Pi + \mathcal S_\alpha \right) \nu_\alpha^{1/2} \right ] \right \}h.
		\end{align*} 
		According to Lemma \ref{lemInvPsi}, $\Psi_{(\lambda,\gamma, \, \alpha)}$ is invertible for small $(\lambda, \gamma)$ and $\alpha \in (\alpha_3, 1]$. Then, provided that $h$ is solution to \eqref{eqfPhi}, we have
		\begin{equation} 
		  \label{eqfPsi}
			h = \Psi_{(\lambda,\gamma, \, \alpha)}^{-1} \Phi_{(\lambda, \gamma, \, \alpha)} \nu_\alpha^{-1/2} \Pi \,\nu_\alpha^{1/2} h .
		\end{equation}
		Let us introduce the ``conjugated operator'' $\mathcal{P} :=\nu_\alpha^{-1/2} \Pi \,\nu_\alpha^{1/2} = \mathbb P \,\nu_\alpha^{1/2}$, where $\mathbb P$ is the projection onto the space of elastic collisional invariants $N_1$.
		We can use it to rewrite  \eqref{eqfPsi} (and then the eigenvalue problem \eqref{pbEigen}) as the following finite dimensional system of equations
		\begin{equation} \label{eqPfPpsi}
			\mathcal P h = \mathcal P \Psi_{(\lambda,\gamma, \, \alpha)}^{-1} \Phi_{(\lambda, \gamma, \, \alpha)} \mathcal P h,
		\end{equation}		 
		which can be understood as: 
		
		\squaredBox{
			\emph{Finding $X_{(\lambda,\gamma, \, \alpha)} = (x_0, \ldots, x_{d+1}) \in \RR^{d+2}$ such that
			\begin{equation*} 
				\left (A_{(\lambda,\gamma, \, \alpha)} - \id\right ) X_{(\lambda,\gamma, \, \alpha)} = 0,
			\end{equation*}
			where $A = \mathcal P \Psi^{-1} \Phi \in \mathcal M_{d+2,\,d+2}(\RR)$.}
		}
		
		\noindent We shall find in the following section some conditions on $\lambda$ in order to have 
		\begin{equation*}
			\det\left (A_{(\lambda,\gamma, \, \alpha)} - \id\right ) = 0.
		\end{equation*}
		Under such conditions, the abstract problem will admit a non-trivial solution $X$. Coming back to the original problem, given that $X = \mathcal P h = \mathbb P \nu_\alpha^{1/2} \, h$, we will obtain thanks to equation \eqref{eqfPsi} a solution $h$ to the original eigenvalue problem \eqref{pbEigen}.
		
	\subsection{Finite Dimensional Resolution} 
	  \label{subFinitDim}
	
		We shall study the vector $X$ component-wise, using the normalization
		\begin{equation} 
		  \label{eqXVect}
			X = \mathbb P \, \nu_\alpha^{1/2} \, h(v) = x_0 \, F_1(v) + (x \cdot v) \, F_1(v) + x_{d+1} \left (|v|^2 - c_\nu\right ) F_1(v),
		\end{equation}
		where $c_\nu$ is such that $\langle \nu_\alpha, |v|^2 \rangle = \langle \nu_\alpha, c_\nu\rangle$.
		If we compute the product of \eqref{eqXVect} with elements of 
		\[\mathcal N_\alpha = \vect \{ \nu_\alpha^{1/2} F_1, \, \nu_\alpha^{1/2} v_i \, F_1 , \, \nu_\alpha^{1/2} \left (|v|^2 -c_\nu\right ) F_1 : \, 1 \leq i \leq d \},\]
		 we find using the definition of $\mathbb P$ and the orthogonality of elements of $N_\alpha$ for $\langle \cdot, \cdot \rangle$ that
		\begin{equation}
		  \label{eqProjection}
			\left\{\begin{aligned}
				& \langle \nu_\alpha^{1/2}h, \,\nu_\alpha^{1/2} F_1 \rangle = x_0 \left \langle \nu_\alpha^{1/2} F_1, \, \nu_\alpha^{1/2} F_1 \right  \rangle, \\ \medskip
				& \langle \nu_\alpha^{1/2} h, \, \nu_\alpha^{1/2} \, v_i \, F_1\rangle = x_i \left \langle \nu_\alpha^{1/2} \, v_i \, F_1, \nu_\alpha^{1/2} \, v_i \, F_1 \right  \rangle, \quad \forall 1 \leq i \leq d, \\ \medskip
				& \langle \nu_\alpha^{1/2} h, \, \nu_\alpha^{1/2} \left (|v|^2-c_\nu\right ) F_1 \rangle = x_{d+1} \left \langle \nu_\alpha^{1/2} \left (|v|^2-c_\nu\right ) F_1, \nu_\alpha^{1/2} \left ( |v|^2-c_\nu \right ) F_1 \right \rangle.
			\end{aligned} \right.
		\end{equation}
		For clarity sake, let us set in the following
		\begin{equation*}
			\langle \phi, \psi \rangle_{F_1} := \int_{\RR^d} \phi(v) \, \overline{\psi}(v) \, F_1(v) \,dv,
		\end{equation*}		
		in order to have
		\[ \left \langle \nu_\alpha^{1/2} \, h_1 \, F_1, \nu_\alpha^{1/2} \, h_2 \, F_1 \right  \rangle = \langle \nu_\alpha h_1, h_2 \rangle_{F_1} .\]
		
		Using \eqref{eqfPsi} together with the relations \eqref{eqProjection}, we obtain for all $1 \leq i \leq d$
		\begin{equation} 
		  \label{eqSystXi}
			\left\{\begin{aligned}
				& x_0 \left \langle \nu_\alpha, 1 \right \rangle_{F_1} = x_0 \left \langle \nu_\alpha, T_\gamma 1 \right \rangle_{F_1} + \left \langle \nu_\alpha, T_\gamma \, (x \cdot v) \right \rangle_{F_1} + x_{d+1} \left \langle \nu_\alpha, T_\gamma \left (|v|^2-c_\nu\right ) \right \rangle_{F_1}, \\ \medskip
				& x_i \left \langle \nu_\alpha \, v_i, v_i \right \rangle_{F_1} = x_0 \left \langle \nu_\alpha \, v_i, T_\gamma 1 \right \rangle_{F_1} + \left \langle \nu_\alpha \, v_i, T_\gamma \, (x \cdot v) \right \rangle_{F_1} + x_{d+1} \left \langle \nu_\alpha \, v_i, T_\gamma \left (|v|^2-c_\nu\right ) \right \rangle_{F_1}, \\ \medskip
				& x_{d+1} \left \langle \nu_\alpha \left (|v|^2-c_\nu \right ), |v|^2-c_\nu \right \rangle_{F_1} = x_0 \left \langle \nu_\alpha \left (|v|^2-c_\nu \right ),T_\gamma 1 \right \rangle_{F_1} + \left \langle \nu_\alpha \left (|v|^2-c_\nu \right ), T_\gamma \, (x \cdot v) \right \rangle_{F_1} \\ \medskip
				& \qquad \qquad \qquad \qquad \qquad \qquad \qquad \ \ + x_{d+1} \left \langle \nu_\alpha \left (|v|^2-c_\nu \right ), T_\gamma \left (|v|^2-c_\nu\right ) \right \rangle_{F_1},
			\end{aligned} \right.
		\end{equation}
		where we have set for fixed $(\lambda, \alpha)$
		\begin{equation} \label{defTgamma}
			T_\gamma := \Psi^{-1}_{(\lambda,\gamma, \,\alpha)}\Phi_{(\lambda, \gamma, \,\alpha)}.
		\end{equation}

		The system \eqref{eqSystXi} is the componentwise version of the projected problem \eqref{eqPfPpsi}.
		We are now going to decompose this system of $d+2$ equations in $X=(x_0, \ldots, x_{d+1})$ in a closed system of $3$ equations in $x_0$, $x \cdot \omega$ and $x_{d+1}$ for a fixed $\omega \in \mathbb{S}^{d-1}$ (corresponding to the \emph{longitudinal} sound waves of the Boltzmann equation, see also the work of Nicolaenko \cite{Nicolaenko:71}) together with a scalar relation in $x_i$ for all $1 \leq i \leq d$ (corresponding to the \emph{transverse} sound waves). 
		For this, we need the following technical lemma:
		
		\begin{lemma}
			Let $x, y \in \mathbb{R}^d$, $\bm e := (1, 0, \ldots, 0)^\mathsf{T}$ and $\gamma = \rho \, \omega$ for $\rho \in \RR$ and $\omega \in \mathbb{S}^{d-1}$. Then, we have
			\begin{gather}
				 \left \langle \nu_\alpha, T_\gamma\, (x \cdot v) \right \rangle_{F_1}  = x \cdot \omega \left \langle \nu_\alpha, T_{\rho \bm{e}} \, v_1 \right \rangle_{F_1}, \label{eqNuTgammav} \\
				 \left \langle \nu_\alpha \,( x \cdot v), T_\gamma \, 1 \right \rangle_{F_1}  = x \cdot \omega \left \langle \nu_\alpha \, v_1, T_{\rho \bm{e}} \, 1 \right \rangle_{F_1}, \label{eqNuvTgamma1} \\
				\begin{split} \left \langle \nu_\alpha \, (x \cdot v), T_\gamma \, (y \cdot v) \right \rangle_{F_1} & = (x \cdot \omega) (y \cdot \omega) \left \langle \nu_\alpha \, v_1, T_{\rho \bm{e}} \, v_1 \right \rangle_{F_1} \\ & + \left[x\cdot \omega - (x \cdot \omega) (y \cdot \omega)\right ] \left \langle \nu_\alpha \, v_2, T_{\rho \bm{e}} \, v_2 \right \rangle_{F_1},\end{split} \label{eqNuvTgammav}  \\
				\left \langle \nu_\alpha \,( x \cdot v), T_\gamma \left (|v|^2-c_\nu \right ) \right \rangle_{F_1}  = x \cdot \omega \left \langle \nu_\alpha \, v_1,T_{\rho \bm{e}}\left (|v|^2-c_\nu \right ) \right \rangle_{F_1}, \label{eqNuvTgammav2} \\
				\left \langle \nu_\alpha \left (|v|^2-c_\nu \right ), T_\gamma\, (x \cdot v) \right \rangle_{F_1}  = x \cdot \omega \left \langle \nu_\alpha \left (|v|^2-c_\nu \right ), T_{\rho \bm{e}} \, v_1 \right \rangle_{F_1}. \label{eqNuv2Tgammav}
			\end{gather}
		\end{lemma}

		\begin{proof}
			According to the definitions of $T_\gamma$ and $\Psi_{(\lambda, \gamma, \,\alpha)}$, the $\gamma$--dependency of $T_\gamma$ is only happening through the operator $\Phi_{(\lambda, \gamma, \,\alpha)}$. 
			But, for $M \in \mathcal{O}(d)$ the orthogonal group of $\mathbb R^d$ (namely,  $M M^* = \bf I_d$) one has for $v \in \mathbb{R}^d$ and $g \in \domain\left (\Phi_{(\lambda, \gamma, \,\alpha)} \right )$, using the fact that $\nu_\alpha$ is a radial function, for all $v \in \RR^d$,
			\begin{align*}
				\left (\Phi_{(\lambda, \gamma, \,\alpha)} \, g\right )(M v) & = \left ( \nu_\alpha(M v) + \lambda + i \, (\gamma \cdot M v) - (1-\alpha) \Delta_v\right )^{-1} \nu_\alpha(M v) \, g(M v), \\
				 & = \left ( \Phi_{(\lambda, M^{-1}\,  \gamma, \,\alpha)} M g\right )(v),
			\end{align*}
			where we have set $Mg\,(v) := g(Mv)$. Then
			\begin{equation*}
				\left (T_\gamma \,g \right ) (M v) = \left (T_{M^{-1} \gamma} \,M g \right ) (v), \quad \forall \, v \in \RR^d.
			\end{equation*}
			Especially, if $g$ is a radial function, there exists a function $\Gamma_g$ such that
			\begin{equation*}
				\left (T_\gamma \, g \right ) (v) = \Gamma_g(\gamma\cdot v, |v|).
			\end{equation*}			
			One has thanks to this result
			\begin{align*}
				\left \langle \nu_\alpha, T_\gamma\, (x \cdot v) \right \rangle_{F_1} & = \left \langle T_\gamma^*\,\nu_\alpha,  (x \cdot v) \right \rangle_{F_1} \\
				 & = \int_{\mathbb{R}^d} \Gamma_{\nu_\alpha}(\gamma\cdot v, |v|) \, (x \cdot v) \, {F_1}(v) \, dv.
			\end{align*}
			Let $M \in \mathcal{O}(d)$ such that $M^{-1} \, \omega = \bm{e}$. Thanks to the change of variables $v = M\xi$ and using the polar coordinates $\gamma = \rho \, \omega$ one has $\gamma \cdot v = \rho \, M^{-1} \, \omega \cdot \xi = \rho \, \xi_1$ and then
			\begin{align*}
				\left \langle \nu_\alpha, T_\gamma\, (x \cdot v) \right \rangle_{F_1} & = \int_{\mathbb{R}^d} \Gamma_{\nu_\alpha}( \rho \,\xi_1, |\xi|) \, (M^{-1} x \cdot \xi) \, {F_1}(\xi) \, d\xi	\\
				 & =  \int_{\mathbb{R}^d} \Gamma_{\nu_\alpha}^{odd}( \rho \, \xi_1, |\xi|) \, (M^{-1} x \cdot \xi)\, {F_1}(\xi) \, d\xi,
			\end{align*}
			where $g^{odd}(a, \cdot) = \left (g(a, \cdot) - g(-a, \cdot)\right )/2$. Given that ${F_1}$ is a radial function of $v$,  $\langle g^{odd}, h \rangle_{F_1} = \langle g, h^{odd} \rangle_{F_1}$ and $(M^{-1} x)_1 = (M^{-1} x) \cdot M^{-1} \omega$, one has
			\begin{align*}
				\left \langle \nu_\alpha, T_\gamma\, (x \cdot v) \right \rangle_{F_1} & =  \int_{\mathbb{R}^d} \Gamma_{\nu_\alpha}^{odd}( \rho \, \xi_1, |\xi|) \, (M^{-1} x)_1 \, \xi_1\, {F_1}(\xi) \, d\xi \\
				 & = (M^{-1} x)_1 \int_{\mathbb{R}^d} \Gamma_{\nu_\alpha}( \rho \, \xi_1, |\xi|) \, \xi_1\, {F_1}(\xi) \, d\xi \\
				 & = x \cdot \omega \left \langle T_{\rho \bm{e}}^*\, \nu_\alpha, v_1 \right \rangle_{F_1},
			\end{align*}
			which proves \eqref{eqNuTgammav}.
			
			Thanks to the same arguments
			\begin{align*}
				\left \langle \nu_\alpha \,( x \cdot v), T_\gamma \, 1 \right \rangle_{F_1} & = \int_{\mathbb{R}^d} \nu_\alpha(v) \, ( x \cdot v) \, \Gamma_1(\gamma\cdot v, |v|) \, {F_1}(v) \, dv \\
				 & = \int_{\mathbb{R}^d} \nu_\alpha(\xi) \, (M^{-1} x \cdot \xi) \, \Gamma_1( \rho \, \xi_1, |\xi|) \, {F_1}(\xi) \, d\xi \\
				 & = \int_{\mathbb{R}^d} \nu_\alpha(\xi) \, (M^{-1} x \cdot \xi) \, \Gamma_1^{odd}( \rho \, \xi_1, |\xi|) \, {F_1}(\xi) \, d\xi \\
				 & = x \cdot \omega  \left \langle \nu_\alpha \, v_1, T_{\rho \bm{e}} 1  \right \rangle_{F_1},
			\end{align*}
			which proves \eqref{eqNuvTgamma1}. Concerning the next identity, one has
			\begin{align*}
				\left \langle \nu_\alpha \, (x \cdot v), T_\gamma \, (y \cdot v) \right \rangle_{F_1} & = \sum_{1 \leq i,j \leq d} x_i \, y_j \int_{\mathbb{R}^d} \nu_\alpha(v) \,v_i \, (T_\gamma \,v_j) (v) \, {F_1}(v) \, dv \\
				 & = \sum_{1 \leq i,j \leq d} x_i \, y_j \int_{\mathbb{R}^d} \nu_\alpha(\xi) \, (M\xi)_i \left (T_{\rho \bm{e}} (M \xi)_j\right ) (\xi) \, {F_1}(\xi) \, d\xi \\
				 & = \sum_{1 \leq i,j,k,l \leq d} x_i \, y_j \, M_{i l} \, M_{j k} \int_{\mathbb{R}^d} \nu_\alpha(\xi) \, \xi_l \left (T_{\rho \bm{e}}\,\xi_k\right ) (\xi) \, {F_1}(\xi) \, d\xi.
			\end{align*}
			If $k \neq l$, this integral is zero (it is clear by doing the transformation $\xi \to -\xi$). In the other case, one has
			\begin{align*}
				\left \langle \nu_\alpha \, (x \cdot v), T_\gamma \, (y \cdot v) \right \rangle_{F_1}  & = \sum_{1 \leq i,j \leq d} x_i \, y_j \, M_{i 1} \, M_{j 1} \int_{\mathbb{R}^d} \nu_\alpha(\xi) \, \xi_1 \left (T_{\rho \bm{e}} \, \xi_1 \right ) (\xi) \, {F_1}(\xi) \, d\xi \\ & + \sum_{\substack{1 \leq i,j \leq d \\ 2 \leq l \leq d}} x_i \, y_j \, M_{i l} \, M_{j l} \int_{\mathbb{R}^d} \nu_\alpha(\xi) \, \xi_2 \left (T_{\rho \bm{e}} \, \xi_2 \right ) (\xi) \, {F_1}(\xi) \, d\xi \\
				  & = (x \cdot \omega) (y \cdot \omega) \left \langle \nu_\alpha \, v_1, T_{\rho \bm{e}} \, v_1 \right \rangle_{F_1} \\ & + \sum_{1 \leq i,j \leq d} x_i \, y_j \left [ (M M^*)_{i j} - M_{i 1}\, M_{j_1}\right ]\left \langle \nu_\alpha \, v_2, T_{\rho \bm{e}} \, v_2 \right \rangle_{F_1},
			\end{align*}
			which is \eqref{eqNuvTgammav} because $M M^* = I_d$. The inequalities \eqref{eqNuvTgammav2} and \eqref{eqNuv2Tgammav} are finally obtained using the same methods of proof.
		\end{proof}
	
		Applying this lemma to system \eqref{eqSystXi}, we find for all $1 \leq i \leq d$
		\begin{gather} 
			x_0 \left \langle \nu_\alpha, T_{\rho \bm{e}} 1  - 1 \right \rangle_{F_1} + x \cdot \omega \left \langle \nu_\alpha, T_{\rho \bm{e}} v_1 \right \rangle_{F_1} + x_{d+1}\left \langle \nu_\alpha, T_{\rho \bm{e}}\left (|v|^2-c_\nu \right ) \right \rangle_{F_1} = 0, \label{eqX0} \\
			\begin{split} x_i \langle \nu_\alpha \, v_i, v_i \rangle_{F_1} & = \omega_i \, x_0 \left \langle \nu_\alpha \, v_1, T_{\rho \bm{e}} 1 \right \rangle_{F_1} + \left[x_i - \omega_i \,(x \cdot \omega) \right] \left \langle \nu_\alpha \, v_2, T_{\rho \bm{e}}v_2 \right \rangle_{F_1}  \\ & + \omega_i \, (x \cdot \omega) \left \langle \nu_\alpha \, v_1, T_{\rho \bm{e}} v_1 \right \rangle_{F_1} + \omega_i \, x_{d+1} \left \langle \nu_\alpha \, v_1, T_{\rho \bm{e}}\left (|v|^2-c_\nu \right ) \right \rangle_{F_1}, 
			\end{split} \label{eqXi} \\
			\begin{split} x_0 \left \langle \nu_\alpha \left (|v|^2-c_\nu \right ), T_{\rho \bm{e}} \, 1 \right \rangle_{F_1} & + x \cdot \omega \left \langle \nu_\alpha \left (|v|^2-c_\nu \right ), T_{\rho \bm{e}} v_1 \right \rangle_{F_1}  \\ & + x_{d+1} \left \langle \nu_\alpha \left (|v|^2-c_\nu \right ), T_{\rho \bm{e}}\left (|v|^2-c_\nu \right ) - \left (|v|^2-c_\nu \right ) \right \rangle_{F_1} = 0. 
			\end{split} \label{eqXd1}
		\end{gather}
	
		Now, on the one hand, if we multiply \eqref{eqXi} by $\omega_i$ and sum over all $i$, using the fact that $|\omega| = 1$, we find that
		\begin{equation} \label{eqXomega}
			x_0 \left \langle \nu_\alpha \, v_1 , T_{\rho \bm{e}} 1 \right \rangle_{F_1} + x \cdot \omega \, \left \langle \nu_\alpha \, v_1, T_{\rho \bm{e}} v_1 - v_1 \right \rangle_{F_1} + x_{d+1}\left \langle \nu_\alpha \, v_1, T_{\rho \bm{e}}\left (|v|^2-c_\nu \right ) \right \rangle_{F_1} = 0.
		\end{equation}
		The system \eqref{eqX0}--\eqref{eqXomega}--\eqref{eqXd1}  is closed in $(x_0, x \cdot \omega, x_{d+1})$ for a fixed $\omega \in\mathbb{S}^{d-1}$. 
		Coming back to a more abstract form, there exists solutions to this system if and only if
		\begin{equation} \label{eqDispRel}
			D(\lambda, \rho, \alpha) = 0,
		\end{equation}
		where we have defined $D$ as the following Gram-like matrix (remember that $T_\gamma $ is given by \eqref{defTgamma} and depends on $(\lambda, \gamma, \alpha)$)
		\begin{multline} \label{defDetlambda}
			D(\lambda, \rho, \alpha) := \\ \begin{vmatrix}
				\left \langle \nu_\alpha, \left (T_{\rho \bm{e}} - \id\right ) 1\right \rangle_{F_1} & \left \langle \nu_\alpha, T_{\rho \bm{e}} v_1 \right \rangle_{F_1} & \left \langle \nu_\alpha, T_{\rho \bm{e}}\left (|v|^2-c_\nu \right ) \right \rangle_{F_1} \\ \medskip
				\left \langle \nu_\alpha \, v_1 , T_{\rho \bm{e}} 1 \right \rangle_{F_1} & \left \langle \nu_\alpha\,  v_1, \left (T_{\rho \bm{e}} - \id\right ) v_1\right \rangle_{F_1} &  \left \langle \nu_\alpha \, v_1, T_{\rho \bm{e}}\left (|v|^2-c_\nu \right ) \right \rangle_{F_1} \\ \medskip
				\left \langle \nu_\alpha\left (|v|^2-c_\nu \right ), T_{\rho \bm{e}} 1 \right \rangle_{F_1} & \left \langle \nu_\alpha\left (|v|^2-c_\nu \right ), T_{\rho \bm{e}} v_1 \right \rangle_{F_1} & \left \langle \nu_\alpha\left (|v|^2-c_\nu \right ), \left (T_{\rho \bm{e}} - \id\right ) \left (|v|^2-c_\nu \right ) \right \rangle_{F_1}
			\end{vmatrix}.
		\end{multline}

		On the other hand, if one multiplies \eqref{eqXomega} by $\omega_i$ and subtract this expression to \eqref{eqXi}, one finds
		\begin{equation} \label{eqXiFinal}
			\left[x_i - \omega_i \, (x \cdot \omega) \right] D_\omega(\lambda, \rho, \alpha) = 0,
		\end{equation}
		where we have set
		\begin{equation} \label{defD0}
			D_\omega(\lambda, \rho, \alpha) := \left \langle \nu_\alpha \,v_1 , \left (T_{\rho \bm{e}} - \id \right ) v_1 \right \rangle_{F_1}.
		\end{equation}
	  Then, if one solves \eqref{eqXomega} in $C \cdot \omega$, the relation \eqref{eqXiFinal} will give the expression of $x_i$, provided that the equation $D_\omega(\lambda, \rho, \alpha) = 0$ admits an unique solution $\lambda$. 	
	
		We will simplify these expressions thanks to the following Lemma.
		\begin{lemma} 
		  \label{lemSimplificationCollInv}
			Let  $(\rho, \lambda,\alpha) \in U_1 \times U_2 \times(\alpha_3,1]$. If $g,h$ are elastic collisional invariants, namely if 
			\[g, h \in N_1 =\vect \{ F_1, \, v_i \, F_1 , |v|^2 \, F_1 : \, 1 \leq i \leq d \},\]
			then we can write for all $\omega \in \SSS^{d-1}$ and $\gamma = \rho \, \omega$
			\begin{gather*}
				\left (\Psi^{-1}_{(\lambda,\gamma, \,\alpha)}\Phi_{(\lambda, \gamma, \,\alpha)} - \id\right ) h =  \Psi^{-1}_{(\lambda, \gamma, \, \alpha)} \left ( \Phi_{(\lambda, \gamma, \,\alpha)} - \id \right )h, \\ 
				\left \langle \nu_\alpha g,\Psi^{-1}_{(\lambda,\gamma, \, \alpha)}\Phi_{(\lambda, \gamma, \,\alpha)} h\right \rangle = \left \langle \nu_\alpha g,  \Psi^{-1}_{(\lambda, \gamma, \, \alpha)} \left ( \Phi_{(\lambda, \gamma, \,\alpha)} - \id \right ) h\right \rangle .
			\end{gather*}
		\end{lemma}
			
		\begin{proof}
			By definition of $\mathcal N_\alpha$, we have $\nu_\alpha^{1/2} h \in \mathcal{N}_\alpha$, and then $\mathcal S_\alpha \, \nu_\alpha^{1/2} h = 0$. 
			But, we know that 
			\[ \Psi_{(\lambda, \gamma, \alpha)} = \id - \Phi_{(\lambda, \gamma, \alpha)}\nu_\alpha^{-1/2} \mathcal S_\alpha \,\nu_\alpha^{1/2}.\]
			Thus, we have $\Psi_{(\lambda, \gamma, \alpha)} h = h$, and given that $\Psi_{(\lambda, \gamma, \, \alpha)}$ is invertible for $(\rho, \lambda,\alpha) \in U_1 \times U_2 \times(\alpha_3,1]$ and $\gamma = \rho \, \omega$, we have
			\begin{equation}
			  \label{eqPsiSimplification}
			  \Psi^{-1}_{(\lambda, \gamma, \alpha)} h = h,
			\end{equation}
			which proves the first relation.
			Using the orthogonality of the collisional invariants and \eqref{eqPsiSimplification}, we obtain the second equality:
			\begin{align*}
			  \left \langle \nu_\alpha g,\Psi^{-1}_{(\lambda,\gamma, \, \alpha)}\Phi_{(\lambda, \gamma, \alpha)} h\right \rangle & = \left \langle \nu_\alpha g,\Psi^{-1}_{(\lambda,\gamma, \, \alpha)}\Phi_{(\lambda, \gamma, \alpha)} h\right \rangle - \left \langle \nu_\alpha g, h\right \rangle \\
			   & = \left \langle \nu_\alpha g, \left (\Psi^{-1}_{(\lambda,\gamma, \, \alpha)}\Phi_{(\lambda, \gamma, \alpha)} - \id \right ) h\right \rangle \\
			   & = \left \langle \nu_\alpha g,  \Psi^{-1}_{(\lambda, \gamma, \, \alpha)} \left ( \Phi_{(\lambda, \gamma, \alpha)} - \id \right ) h\right \rangle .
			\end{align*}
		\end{proof}
		
	  Let us set $ \Upsilon_{(\lambda, \gamma, \, \alpha)} := \Psi^{-1}_{(\lambda, \gamma, \, \alpha)} \left ( \Phi_{(\lambda, \gamma, \, \alpha)} - \id \right )$. Thanks to this lemma, to the definition of $c_\nu$ and by the nullity of the odd moments of the centered Gaussian $F_1$, we can write \eqref{defDetlambda} in a ``simpler'' form, namely 
		\begin{equation}
		  \label{defDetlambdaSimp}
			D(\lambda, \rho, \alpha) = \begin{vmatrix}
				\left \langle \nu_\alpha, \Upsilon_{(\lambda, \, \rho \bm{e}, \, \alpha)} \, 1\right \rangle_{F_1} & \left \langle \nu_\alpha, \Upsilon_{(\lambda, \, \rho \bm{e}, \, \alpha)} \, v_1 \right \rangle_{F_1} & \left \langle \nu_\alpha, \Upsilon_{(\lambda, \, \rho \bm{e}, \, \alpha)} \,g \right \rangle_{F_1} \\
				\left \langle \nu_\alpha \, v_1 , \Upsilon_{(\lambda, \, \rho \bm{e}, \, \alpha)} \, 1 \right \rangle_{F_1} & \left \langle \nu_\alpha\,  v_1, \Upsilon_{(\lambda, \, \rho \bm{e}, \, \alpha)} \, v_1\right \rangle_{F_1} &  \left \langle \nu_\alpha \, v_1, \Upsilon_{(\lambda, \, \rho \bm{e}, \, \alpha)} \,g \right \rangle_{F_1} \\
				\left \langle \nu_\alpha \,g,\Upsilon_{(\lambda, \, \rho \bm{e}, \, \alpha)} \, 1 \right \rangle_{F_1} & \left \langle \nu_\alpha \,g, \Upsilon_{(\lambda, \, \rho \bm{e}, \, \alpha)} \, v_1 \right \rangle_{F_1} & \left \langle \nu_\alpha \,g, \Upsilon_{(\lambda, \, \rho \bm{e}, \, \alpha)} \,g \right \rangle_{F_1}
			\end{vmatrix},
		\end{equation}
		where we have set $g(v) := |v|^2-c_\nu$. We can also write \eqref{defD0} the same way
		\begin{equation} 
		  \label{defD0Simp}
			D_\omega(\lambda, \rho, \alpha) = \left \langle \nu_\alpha \,v_1 , \Upsilon_{\rho \bm{e}} \, v_1 \right \rangle_{F_1}.
		\end{equation}		
    Before solving these equations, we need a last lemma.
		
		\begin{lemma} 
		  \label{lemPsi0}
			Let $h$ be an elastic collisional invariant (namely $h \in N_1$). If $\Psi^*$ denotes the adjoint operator of $\Psi$, then we have
			\begin{align*}
				\left (\Psi^*_{(0, 0, 1)}\right )^{-1} \, \nu_1 h  = \nu_1 h.
			\end{align*}
		\end{lemma}

		\begin{proof} 
			Let $ (\rho, \lambda,\alpha) \in U_1 \times U_2 \times(\alpha_2,1]$ and $\omega \in \mathbb{S}^{d-1}$ and set $\gamma = \rho \, \omega$.
      If $T$ is an invertible operator on a Banach space, it is known that $(T^*)^{-1} = (T^{-1})^*$. Moreover, provided that $\nu_1 \in \RR$, the adjoint  operator $\Phi^*_{(\lambda, \gamma, \, 1)}$ is the operator of multiplication by 
			\begin{equation*}
				\frac{\nu_1(v)}{ \nu_1(v) + \bar \lambda - i \, (\gamma \cdot v)}
			\end{equation*}			
			Then, if $(\lambda, \rho) \to (0,0)$ we have $\Phi_{(\lambda, \rho \,\omega, \, 1)}^* \to \id$ strongly.
      But, we can also compute 
			\begin{align*}
			  \Psi_{(\lambda, \, \rho\, \omega,\, \alpha)}^* & = \id \, - \, \left (\Phi_{(\lambda, \rho \, \omega)} \nu_1^{-1/2} \mathcal S_1 \, \nu_1^{1/2}\right )^* \\
			  & =  \id \, - \, \nu_1^{1/2} \mathcal S_1^* \,\nu_1^{-1/2} \, \Phi_{(\lambda, \rho \, \omega)}^*,
			\end{align*}			
			and as $h \in N_1$,  we have 
			\[\nu_1^{1/2} \mathcal S_1^* \, \nu_1^{- 1/2} \, \nu h  =  0. \]
			Finally, we can write
			\begin{align*}
				\Psi_{(0, 0, 1)}^* \nu_1 h & = \left ( \id \, - \, \nu_1^{1/2} \mathcal S_1^* \, \nu_1^{-1/2} \right ) \nu_1 h \\
				 & = \nu_1 h,
			\end{align*}
			which concludes the proof after inversion.
		\end{proof}

		\begin{remark}
		  \label{remAnalyticEig}
		  The eigenvalues and eigenvectors of $\LL_{\alpha, \,\rho\omega}$ are analytic function or $\rho$. Indeed, thanks to the hard spheres kernel and estimates \eqref{nuCrois},  there exists a nonnegative constant $M$ such that
		  \[ \left \|(\omega \cdot v) \,h\right \|_{L^1(m^{-1})} \leq M \left ( \left \|h\right \|_{L^1(m^{-1})}+\left \| \LL_\alpha h\right \|_{L^1(m^{-1})}\right ). 
		  \]
		  We can then apply \cite[Thm. VII.2.6 and Rem. VII.2.7]{kato:1966} about the analyticity of the spectrum of a closed operator on a Banach space.
		\end{remark}

	\subsection{First Order Coefficients of the Taylor Expansion}
		
		We can now study in details for what values of the parameters $\lambda$ and $\alpha$ one can solve the projected eigenvalue problem \eqref{eqXiFinal}. We start by considering the behavior of the transverse sound waves.
		
		
		\begin{proposition} 
		  \label{propD0}
			Let $\omega \in \SSS^{d-1}$. There exist $\rho_0 > 0$ and $\alpha_4 \in (\alpha_3, 1]$ such that the problem of solving the equation
			\[D_\omega(\lambda, \rho, \alpha) = 0\] 
			has a unique solution $\lambda_\omega = \lambda_\omega(\rho, \alpha) \in \mathcal{C}^{\infty} \left ( (-\bar \rho_0, \bar \rho_0) \times (\alpha_4, 1]\right )$, verifying
			\begin{equation*}
			  \lambda_\omega(0,1) = \frac{\partial \lambda_\omega}{\partial \rho} (0,1) = \frac{\partial \lambda_\omega}{\partial \alpha} (0,1) = 0.  
			\end{equation*}
		\end{proposition}
		
		\begin{proof}
			Let us write thanks to the compact expression \eqref{defD0Simp} of $D_\omega$
			\begin{align*}
				0 & = - D_\omega(\lambda, \rho, \alpha) \\	
				 & = - \left \langle \left (\Psi_{(\lambda, \, \rho \bm{e}, \, \alpha)}^*\right )^{-1} ( \nu_\alpha v_1) , \left( \Phi_{(\lambda, \, \rho \bm{e}, \, \alpha)} - \id \right) v_1 \right \rangle_{F_1} \\
				 & = -\int_{\mathbb{R}^d} \left (\Psi_{(\lambda, \, \rho \bm{e}, \, \alpha)}^* \right )^{-1} ( \nu_\alpha v_1) \left [  \left ( \nu_\alpha(v) + \lambda + i \, (\rho \bm{e} \cdot v)  - (1-\alpha) \Delta_v\right )^{-1} \nu_\alpha(v) - \id \right ] (v_1) \,{F_1}(v) \, dv \\
				 & = \int_{\mathbb{R}^d} \left (\Psi_{(\lambda, \, \rho \bm{e}, \, \alpha)}^* \right )^{-1} ( \nu_\alpha v_1) \left ( \nu_\alpha(v) + \lambda + i \, \rho \, v_1  - (1-\alpha) \Delta_v\right )^{-1} \left ( \lambda + i \, \rho \, v_1  - (1-\alpha) \Delta_v\right ) (v_1) \\
				 & \qquad \qquad\qquad\qquad\qquad\qquad\qquad\qquad \qquad\qquad\qquad\qquad\qquad\qquad\qquad\qquad\qquad \qquad{F_1}(v) \, dv.
			\end{align*}
			Let us now set $z = \lambda/\rho$ and $s =  1-\alpha $. We shall take the limit $(\rho, s) \to (0, 0)$ in $D_\omega$. For this, we define a new function $G_\omega$ as
			\begin{equation*}
				G_\omega(z, \rho, s) := \frac{1}{\rho} D_\omega(\rho z, \rho, 1- s).
			\end{equation*}
			Then, as $\Delta_v (v_1) = 0$, we will have $D_\omega(\lambda, \rho, \alpha) = 0$ if and only if
			\begin{align*}
				0 & = -G_\omega(z, \rho, s) \\
				 & = \int_{\mathbb{R}^d} \left (\Psi_{(\rho z, \rho \bm{e}, 1- s)}^*\right )^{-1} ( \nu_{1- s} v_1) \left ( \nu_{1- s}(v) + \rho z + i \, \rho \, v_1  - s \Delta_v\right )^{-1} \left ( ( z + i v_1 ) \, v_1 \right ) {F_1}(v) \, dv.
			\end{align*}
			
			Moreover, if $\alpha \to 1$, thanks to the continuity of the equilibrium profiles $F_\alpha$ with respect to $\alpha$ (recalled in Proposition \ref{propCvMaxwProfil}) and to the smoothness properties of these profiles (recalled in Proposition \ref{propRegSelfSimProf}), we have $\nu_{\alpha}(v) \to \nu_1(v)$, uniformly in $v$. 
			Hence, if we take the limit $(\rho,s) \to (0,0)$, we find thanks to Lemma \ref{lemPsi0} that
			\begin{align*}
				0 & = - G_\omega(z, 0, 1) \\
				 & = \int_{\mathbb{R}^d} \left (\Psi_{(0, \, 0, \, 1)}^*\right )^{-1} ( \nu_1 v_1) \frac{z + i v_1}{\nu_1} (v_1) \, {F_1}(v) \, dv \\
				 & = z \int_{\mathbb{R}^d} v_1^2  \, {F_1}(v) \, dv = z \, \bar T_1.
			\end{align*}
			Provided that $\bar T_1$ is nonzero, we have $z = 0$.
			
			It just remains to apply the implicit function theorem to the map $(z, \rho, s) \mapsto G_\omega(z, \rho,s)$ in $(0,0,0)$. Provided that we have 
			\begin{equation*}
				\left \{ \begin{aligned}
					& G_\omega(0,0,0) = 0, \\
					& \frac{\partial G_\omega}{\partial z}(0,0,0) = {\bar T_1},
				\end{aligned} \right.
			\end{equation*}
			there exist two real constants $\bar \rho_0 > 0$, $\alpha_4 \in (\alpha_3, 1]$ and a mapping $z_\omega \in \mathcal{C}^\infty\left ((-\bar \rho_0, \bar \rho_0) \times [0, 1-\alpha_4)\right )$ such that if $|\rho| \leq \bar \rho_0$ and $ s \in [0, 1-\alpha_4)$, then
			\begin{equation*}
				\frac{1}{\rho} D_\omega\left (\rho z_\omega(\rho,s), \rho, 1- s \right ) = G_\omega\left( z_\omega(\rho,s), \rho, s\right ) = 0.
			\end{equation*}
			To conclude the proof, we set $\lambda_\omega(\rho,\alpha) := \rho z_\omega\left ( \rho,1-\alpha \right )$ and this function has the properties we were looking from.
		\end{proof}

		Let us now turn to the dispersion relations \eqref{eqDispRel}, corresponding to the longitudinal sound waves. 
		We recall the simplified expression of $D$ for the reader convenience:
    \begin{equation*}
			D(\lambda, \rho, \alpha) = \begin{vmatrix}
				\left \langle \nu_\alpha, \Upsilon_{(\lambda, \, \rho \bm{e}, \, \alpha)} \, 1\right \rangle_{F_1} & \left \langle \nu_\alpha, \Upsilon_{(\lambda, \, \rho \bm{e}, \, \alpha)} \, v_1 \right \rangle_{F_1} & \left \langle \nu_\alpha, \Upsilon_{(\lambda, \, \rho \bm{e}, \, \alpha)} \,g \right \rangle_{F_1} \\
				\left \langle \nu_\alpha \, v_1 , \Upsilon_{(\lambda, \, \rho \bm{e}, \, \alpha)} \, 1 \right \rangle_{F_1} & \left \langle \nu_\alpha\,  v_1, \Upsilon_{(\lambda, \, \rho \bm{e}, \, \alpha)} \, v_1\right \rangle_{F_1} &  \left \langle \nu_\alpha \, v_1, \Upsilon_{(\lambda, \, \rho \bm{e}, \, \alpha)} \,g \right \rangle_{F_1} \\
				\left \langle \nu_\alpha \,g,\Upsilon_{(\lambda, \, \rho \bm{e}, \, \alpha)} \, 1 \right \rangle_{F_1} & \left \langle \nu_\alpha \,g, \Upsilon_{(\lambda, \, \rho \bm{e}, \, \alpha)} \, v_1 \right \rangle_{F_1} & \left \langle \nu_\alpha \,g, \Upsilon_{(\lambda, \, \rho \bm{e}, \, \alpha)} \,g \right \rangle_{F_1}
			\end{vmatrix}.
		\end{equation*}
		We prove the following result concerning the behavior of the eigenvalues for small frequency and inelasticity.
		  
		\begin{proposition} 
		  \label{propD}
			For $\lambda \in U_2$ (see Lemma \ref{lemInvPsi}), there exists $\bar \rho > 0$ and $\alpha_5 \in (\alpha_4,1]$ such that for $\alpha \in (\alpha_5,1]$ the elastic dispersion relation $D(\lambda, \rho, \alpha ) = 0$ has exactly three branches of solutions $\lambda^{(j)}(\rho, \alpha)$ for all $j \in \{-1,0,1\}$ and $\rho \in (-\bar \rho_1, \bar \rho_1)$. 
			These solutions are of class $\mathcal{C}^{\infty}(-\bar \rho_1, \bar \rho_1)$ and verify
			\begin{equation*}
				\left\{ \begin{aligned} 
					& \lambda^{(j)}(0,1) = 0, && \forall \, j \in \{-1,0,1\}, \\
					& \frac{\partial \lambda^{(j)}}{\partial \rho}(0,1) = j \, i \, \sqrt{\bar T_1 + \frac{2\bar T_1^2}{d}}, && \forall \, j \in \{-1,0,1\}, \\
					& \frac{\partial \lambda^{(0)}}{\partial \alpha}(0,1) = - \frac{3}{\bar T_1},
				\end{aligned}\right.
			\end{equation*}
			where, $ \lambda^{(0)}$ is the so-called \emph{energy eigenvalue} and $\bar T_1$ is given by \eqref{defElasticEquilibTemperature}.
			Finally, we also have the symmetry properties
			\begin{equation}
			  \label{eqSymRel}
			  \lambda^{(j)}(-\rho, \alpha) = \overline{\lambda^{(j)}}(\rho, \alpha) = \lambda^{(-j)}(\rho, \alpha).
			\end{equation}
		\end{proposition}

		\begin{proof}
		  We shall use the ideas introduced in the proof of Proposition \ref{propD0}: instead of solving directly the equation \eqref{eqDispRel}, we want to solve an equivalent one depending on $z = \lambda/\rho$ and we set to simplify $s =  1-\alpha $. We then introduce a function $G = G(z, \rho, s)$ by setting
			\begin{equation*}
				G(z, \rho, s) := \frac{1}{\rho^3} D(\rho z, \rho, 1- s). 
			\end{equation*}
			According to the simplified expression \eqref{defDetlambdaSimp} of $D$, all the components of the matrix found in $G(z, \rho, s)$ can be written for $h_1, \ h_2 \in N_1$
			\begin{multline*}
			   \frac 1 \rho \left \langle \nu_{1-s} \, h_1 , \Upsilon_{(\rho z, \, \rho \bm{e}, \, 1-s)} \, h_2\right \rangle_{F_1} = \int_{\mathbb{R}^d} \left (\Psi_{(\rho z, \, \rho \bm{e}, \, 1- s)}^*\right )^{-1} ( \nu_{1- s} h_1)(v) \\
			   \left ( \nu_{1- s}(v) + \rho z + i \, \rho \, v_1  - s \Delta_v\right )^{-1} \left ( z + i v_1 - \frac s \rho \Delta_v \right )( h_2)(v) F_1(v) \, dv.
			\end{multline*}
			By doing the same computations than in the proof of Proposition \ref{propD0}, this quantity becomes for $\rho = s = 0$
			\begin{equation*}
				\int_{\mathbb{R}^d}  h_1(v) ({z + i v_1 }) h_2 (v) \, F_1(v)  \, dv.
			\end{equation*}					
			Moreover, according to the definition of the Maxwellian distribution $F_1$, we have
			\begin{equation*}
			  \int_{\mathbb{R}^d} \begin{pmatrix} 1 \\ |v|^2 \\ v_1^2 \, |v|^2 \end{pmatrix} F_1(v) \, dv
			  = \begin{pmatrix} 1 \\ \bar T_1 \\ \left (d+2\right ) \bar T_1^2  \end{pmatrix}.
			\end{equation*}	
			Thus, we can write $D$ as
			\begin{align*}
				G(z, 0, 1) & = \begin{vmatrix}
					\langle 1, z + i v_1 \rangle_{F_1} & \langle 1, (z + i v_1) v_1 \rangle_{F_1}  & \langle 1, (z + i v_1) g\rangle_{F_1} \\
					\langle v_1, z + i v_1 \rangle_{F_1} & \langle v_1, (z + i v_1) v_1 \rangle_{F_1} & \langle v_1, (z + i v_1) g \rangle_{F_1} \\
					\langle g, z + i v_1 \rangle_{F_1} & \langle g, (z + i v_1) v_1 \rangle_{F_1} & \langle g, (z + i v_1) g \rangle_{F_1} \\	
				\end{vmatrix} \\
				 & = \begin{vmatrix}
					z          & i \, \bar T_1  & z \left ( d \, \bar T_1 - c_\nu \right ) \\
					i \, \bar T_1 & z \, \bar T_1  & i \, \bar T_1 \left ( \left (d+2\right )\bar T_1  - c_\nu \right ) \\
					z \left ( d \, \bar T_1 - c_\nu \right ) & i \, \left ( \left (d+2\right )\bar T_1 - c_\nu \right ) & z \left ( \left ( d \, \bar T_1 - c_\nu\right )^2 + 2 d \, \bar T_1^2 \right ) 
				\end{vmatrix}  \\
			  & = 2   \, \bar T_1^2 \, z \left (d z^2 +d\,\bar T_1+2 \, \bar T_1^2\right ) \\			  
				& = 2 d \, \bar T_1^2 \, (z-z_{-1}) ( z - z_0) \, (z-z_{+1}),
			\end{align*}
			where we have set for any $j \in \{-1,0,+1\}$
			\[z_j := j \, i \, \sqrt{\bar T_1 + \frac{2 \bar T_1^2}{d}}.\] 
			Hence, provided that $G(z, 0, 1)$ has no multiple root, we have shown that
			\begin{equation*}
				\left \{ \begin{aligned}
					& G(z_j, 0,0) = 0, \\
					& \frac{\partial G}{\partial z}(z_j, 0,0) \neq 0,
				\end{aligned} \right.
			\end{equation*}
			and we can apply again the implicit function theorem to show that in a neighborhood $\mathcal B \times (-\bar \rho_1, \bar \rho_1) \times [0, 1 - \alpha_5)$ of $(z_j, 0)$, there exists an unique function $\widetilde{z_j} \in \mathcal{C}^{\infty}\left ((-\bar \rho_1, \bar \rho_1) \times [0, 1 - \alpha_5)\right )$ such that if $|\rho| \leq \bar \rho_1$ and $s \in [0, 1 - \alpha_5)$, then $G(\widetilde{z_j}(\rho,s), \rho, s) = 0$ (and of course $\widetilde{z_j}(0,0) = z_j$). 
			We finally set $\lambda^{(j)}(\rho,s) := \rho \widetilde{z_j}(\rho,s)$, which give a solution to  \eqref{eqDispRel} (with $\alpha = 1-s$) verifying
			\begin{equation*}
				\left \{ \begin{aligned}
					& \lambda^{(j)}(0,0) = 0, \\
					& \frac{\partial \lambda^{(j)}}{\partial \rho}(0,0) = z_j. \\
				\end{aligned} \right .
			\end{equation*}
			We now have to prove that these three branches are the only solutions to $D(\lambda, \rho, \alpha) = 0$ for small $\rho$ and $1-\alpha$.
			
			For this, following once more \cite{ellis:1975}, we shall use tools from complex analysis. Let us fix $|\rho| \leq \bar \rho_1$ and $\alpha \in (\alpha_5,1]$; according to the definition of $D$, the map $\lambda \mapsto D(\lambda, \rho,\alpha)$ is holomorphic on the set $U_2$ (defined Lemma \ref{lemInvPsi}). 
			Moreover, following the previous computations, we also have $D(\lambda, 0, 1) = \lambda^3 H(\lambda)$ for a function $H$ holomorphic on $U_2$ such that $H(0) = 1$. Hence, if $\lambda$ is defined along a circle $\mathcal C$ around $0$, then $D(\lambda, 0,1)$ will encircle the origin exactly three times. 
			Using the strong convergence of the multiplication operator $\Phi_{(\lambda, \, \rho \bm{e}, \, \alpha)}$ towards $\id$ when $(\rho, \, \alpha) \to (0,1)$, we can write
			\begin{equation*}
				\lim_{(\rho, \, \alpha) \to (0,1)} \sup_{\lambda \in U_2} | D(\lambda, \rho, \alpha) - D(\lambda, 0, 1)| = 0.
			\end{equation*}
			Hence, for small $(\rho, 1- \alpha)$, the function $\lambda \mapsto D(\lambda, \rho, \alpha)$ encircles the origin also only three times when $\lambda$ traverses $\mathcal C$. This function then only has three roots for fixed $\rho$ and $\alpha$.
			
			Next, we compute the partial derivative with respect to $\alpha$ of the energy eigenvalue. This eigenvalue is given by the solution of the dispersion relation that depends on $\rho$ only at second order, namely $\lambda^{(0)}$. 
			Let $h^{(0)}_{(\rho \, \omega, \, \alpha)}$ be the associated eigenvector. We then have for all $\omega \in \SSS^{d-1}$ and $\rho \geq 0$
			\begin{equation}
			  \label{defEnergyEigenEq}
			  \LL_{(\alpha,\gamma)} \, h^{(0)}_{(\rho \, \omega, \, \alpha)}(v) \, = \, \lambda^{(0)}(\rho,\alpha) \, h^{(0)}_{(\rho \, \omega, \, \alpha)}(v), \quad \forall \, v \in \RR^d.
			\end{equation}
			In particular, the ``elastic, space homogeneous'' energy eigenvector $h^{(0)}_{(0,1)}$ is defined thanks to the Maxwellian profile $F_1$  (given in \eqref{defElasticEquilibTemperature}) as
			\begin{equation}
			  \label{eqEigenValEnergElastic}
			  h^{(0)}_{(0,1)} = c_0 \left ( |v|^2 - d \, \bar T_1\right ) F_1,
			\end{equation}
			where $c_0$ is a normalizing constant. We have by construction, using some elementary properties of Gaussian functions 
			\[ \left \| h^{(0)}_{(0,1)} \right \|_{L^1(m^{-1})} = 1, \quad N\left ( h^{(0)}_{(0,1)}\right ) = 0, \quad \mathcal E \left ( h^{(0)}_{(0,1)}\right ) = 2 \, c_0 \, d \, \bar T_1^2, \]
			where we have defined the \emph{mass} $N(f)$ and the \emph{kinetic energy} $\mathcal E(f)$ of a given distribution $f$ as
			\[ N(f) := \int_{\RR^d} f(v) \, dv, \quad \mathcal E(f) := \int_{\RR^d} f(v) \, |v|^2 \, dv.\]
			
			By integrating the eigenvalue equation \eqref{defEnergyEigenEq} against $|v|^2$ we obtain according to the expression of the energy dissipation functional \eqref{defDissipFunctional}
			\begin{multline*}
			  \lambda^{(0)}(\rho,\alpha) \mathcal{E}\left (h^{(0)}_{(\rho \, \omega, \, \alpha)}\right ) = \\
			  -2 (1-\alpha^2) \, D\left (F_\alpha,  h^{(0)}_{(\rho \, \omega, \alpha)}\right ) + 2 d N(1-\alpha)\left (h^{(0)}_{(\rho \, \omega, \, \alpha)}\right ) + i {\rho} \omega \cdot \int_{\RR^d} h^{(0)}_{(\rho \, \omega, \, \alpha)}(v) \,v \, |v|^2 \, dv.
			\end{multline*}
			As $\rho$ tends to $0$, dividing by $1-\alpha$ yields
			\begin{equation*}
			  \frac{\lambda^{(0)}(0,\alpha)}{1-\alpha} \mathcal{E}\left (h^{(0)}_{(0, \alpha)}\right ) = -2 (1+\alpha) \, D\left (F_\alpha,  h^{(0)}_{(0, \alpha)}\right ) + 2 d N\left (h^{(0)}_{(0, \alpha)}\right ).
			\end{equation*}
			Now, we use the rate of convergence of the inelastic profile $F_\alpha$ towards the elastic one $F_1$ recalled in Proposition \ref{propCvMaxwProfil} and the smoothness of $ h^{(0)}_{(0, \alpha)}$ with respect to $\alpha$ obtained thanks to the use of the implicit functions theorem. We then obtain thanks to the nullity of the mass of $h^{(0)}_{(0,1)}$
			\begin{equation}
			  \label{eqLambda10alpha}
			  \frac{\lambda^{(0)}(0,\alpha)}{1-\alpha} \mathcal{E}\left (h^{(0)}_{(0, 1)}\right ) = -2 (1+\alpha) \, D\left (F_1,  h^{(0)}_{(0,1)}\right ) + \mathcal O(1-\alpha).
			\end{equation}
			Finally, we compute thanks to the expression of the elastic energy eigenvector \eqref{eqEigenValEnergElastic} and to the definition \eqref{defElasticEquilibTemperature} of the equilibrium temperature the quantities
			\[ \mathcal{E}\left (h^{(0)}_{(0, 1)}\right ) = 2 \, d \, c_0 \, \bar T_1^2, \quad D\left (F_1,  h^{(0)}_{(0,1)}\right ) = \frac 32 d \, c_0 \,  \bar T_1.\]
			Gathering these relations and passing to the limit $\alpha \to 1$ in \eqref{eqLambda10alpha} gives the result.

			Concerning the last assertion of the proposition, we notice thanks to the invariance of the eigenvalue problem \eqref{pbEigen} under the composition of the convex conjugation and the reflection $\gamma \to -\gamma$ that $\overline D(\lambda, \rho, \alpha) = D(\overline \lambda, -\rho, \alpha) = D(\overline \lambda, \rho, \alpha)$.
		\end{proof}
		
			\begin{remark}
			  As a consequence of the symmetry  relation \eqref{eqSymRel}
			  \[\lambda^{(j)}(-\rho, \alpha) = \overline{\lambda^{(j)}}(\rho, \alpha) = \lambda^{(-j)}(\rho, \alpha),\]
			  we have $\lambda^{(0)}(\rho, \alpha) \in \RR$.
			\end{remark}

			Thanks to this proposition, we can construct the $d+2$ normalized hydrodynamic eigenvectors
			\[\left (h^{(j)}_{(\rho\,\omega, \, \alpha)} \right )_{j \in \{-1, \ldots, d\}}\]
			of the inelastic linearized collision operator, for small $\rho$,  $\alpha$ close to $1$ and a given $\omega \in \SSS^{d-1}$. 
			Indeed, on the one hand, for $j \in \{2, \ldots, d\}$, we take $\lambda = \lambda_\omega(\rho, \alpha)$ for $|\rho| \leq \bar \rho_0$ and $\alpha \in (\alpha_4,1]$ given by Proposition \ref{propD0} and choose in \eqref{eqXVect} $x_0 =x_{d+1} = 0$ and any vector $x \in \omega^\perp$.
	    The relation \eqref{eqfPsi} then allows us to construct the eigenvectors $h^{(j)}_{(\rho, \,\omega, \, \alpha)}$ associated to the conservation of momentum.
	    	
	    On the other hand, for $j \in \{-1, 0, 1 \}$, we pick a solution $\lambda = \lambda^{(j)}(\rho, \, \alpha)$ for $|\rho| \leq \bar \rho_1$ and $\alpha \in (\alpha_5,1]$ to the dispersion relation $D(\lambda, \rho, \alpha)=0$ given by Proposition \ref{propD} and choose the vector $\left (x_0, x \cdot \omega, x_{d+1}\right )$ to be a solution to the system \eqref{eqX0}--\eqref{eqXomega}--\eqref{eqXd1} corresponding to this eigenvalue. 
	    We then set $x = (x \cdot \omega) \, \omega$ and recover through \eqref{eqXVect}
	    \[ \mathcal P h^{(j)}_{(\rho\,\omega, \, \alpha)}(v) = x_0(\rho, \, \alpha) + (x \cdot v)(\rho, \omega \cdot v, \, \alpha) + x_{d+1}(\rho, \, \alpha) \left (|v|^2 - c_\nu\right ).\]
	    Inserting this expression in \eqref{eqfPsi} finally gives us the eigenvalue, depending on $\rho$, $\alpha$ (as a $\mathcal C^\infty$ function), $|v|$ and $v \cdot \omega$.
	    With this procedure, we have constructed three independent solutions (corresponding to the acoustic waves and the kinetic energy) $h^{(j)} = h^{(j)}_{(\rho\, \omega, \, \alpha)} \in L^1(m^{-1})$ to the eigenvalue problem
	    \[ \left (- i\rho (\omega \cdot v) + \mathcal{L}_\alpha \right ) h^{(j)} = \lambda^{(j)} \,h^{(j)}, \quad \forall \, j \in \{-1, 0, 1 \} .\]

	  \subsection{Higher Order Expansion}
	    \label{subHighOrder}
	    
	    We are interested in this section to give an expression for the expansion of the eigenvalues with respect to the spatial coordinate $\gamma = \rho \, \omega$.
	    We have seen in Remark \ref{remAnalyticEig} that for a fixed $\omega \in \SSS^{d-1}$, the eigenvalues $\lambda^{(j)}(\rho, \, \alpha)$ and eigenvectors $h^{(j)}_{(\rho\,\omega, \, \alpha)}$ are analytic functions of the radial coordinate $\rho$ and the inelasticity $1-\alpha$. 
	    Hence, we have for any $v \in \RR^d$
	    \begin{align}
	      & \lambda^{(j)}(\rho, \, \alpha) = \sum_{n \geq 0} \lambda^{(j)}_n \rho^n + (1-\alpha) e^{(j)}_1 + \mathcal{O}\left ((1-\alpha)^2 + (1-\alpha)\rho\right ), \label{eqDevelLambda} \\ 
	      & h^{(j)}_{(\rho\, \omega, \, \alpha)}(v) = \sum_{n \geq 0} h^{(j)}_n(\omega)(v) \rho^n + (1-\alpha) f^{(j)}_1(v) + \mathcal{O}\left ((1-\alpha)^2 + (1-\alpha)\rho\right ). \label{eqDevelH}
	    \end{align}
	    According to the computations of the previous subsection, the first order components of this expansion are given by
			\begin{equation*}
				\left\{ \begin{aligned} 
					& \lambda^{(j)}_0 = 0, && \forall \, j \in \{-1,\ldots,d\}, \\
					& \lambda^{(j)}_1 = j \, i \, \sqrt{\bar T_1 + \frac{2\bar T_1^2}{d}}, && \forall \, j \in \{-1,0,1\}, && \lambda^{(j)}_1 = 0, && \forall \, j \in \{2,\ldots,d\}, \\
					& e_1^{(0)} = - \frac{3}{\bar T_1}, && && e_1^{(j)} = 0,&& \forall \, j \in \{-1,1,\ldots,d\}.
				\end{aligned}\right.
			\end{equation*}
			We also have for any $v \in \RR^d$
			\[ h^{(0)}_0(v) = c_0 \left ( |v|^2 - d \, \bar T_1\right ) F_1,\]
			for a nonnegative normalizing constant $c_0$.
	    Since the triple $\left (\rho \, \omega, \lambda^{(j)}(\rho, \, \alpha), h^{(j)}_{(\rho\, \omega, \, \alpha)}\right )$ is solution to the eigenvalue problem \eqref{pbEigenValueLambdaGammaRho}, we can equate the power of $\rho$ and $1-\alpha$ in \eqref{eqDevelLambda}--\eqref{eqDevelH} to obtain
	    \begin{equation}
	      \label{eqExpansionRecurs}
	      \left \{ \begin{aligned}
	        & \LL_0 \, h^{(j)}_0(\omega) = 0, && \forall \, j \in \{-1,\ldots,d\},\\
	        & \LL_0 \, h^{(j)}_1(\omega) = \left (\lambda^{(j)}_1 + i (\omega \cdot v)\right )h^{(j)}_0(\omega), && \forall \, j \in \{-1,\ldots,d\},\\
	        & \LL_0 \, h^{(j)}_n(\omega) = \left (\lambda^{(j)}_1 + i (\omega \cdot v)\right )h^{(j)}_{n-1}(\omega) + \sum_{k=2}^{n} \lambda^{(j)}_k h^{(j)}_{n-k}(\omega),  && \forall \, j \in \{-1,\ldots,d\}, \ n \geq 2,
	      \end{aligned}\right.
	    \end{equation}
	    where we also used the smoothness of $\LL_\alpha$ with respect to $1-\alpha$ (Proposition \ref{propHoldContOp}).
	    
	    Hence, the coefficients of the expansion can be computed by induction. For example, to compute $\lambda^{(j)}_2$, we can integrate the eigenvalue problem \eqref{pbEigenValueLambdaGammaRho} with respect to $|v|^2$ and use the equations \eqref{eqExpansionRecurs} with $n=2$ to obtain
	    \begin{equation}
	      \label{eqLambdajExplicit}
	      \lambda^{(j)}_2 = -\frac{i}{2 \, d \, c_0 \, \bar T_1^2} \, \omega \cdot {q\left (h^{(j)}_1(\omega)\right )},
      \end{equation}	     
      where we have set
      \[ q(h) := \int_{\RR^d} h(v) \, v \, |v|^2 \, dv.\]
      Now, using again \eqref{eqExpansionRecurs} for $n=1$, we know that
      \[\LL_0 \, h^{(j)}_1(\omega) = \left (\lambda^{(j)}_1 + i (\omega \cdot v)\right )h^{(j)}_0(\omega). \]
      Since $\lambda^{(j)}_1$ is an imaginary number and $h^{(j)}_0$ a real number (it is the elastic, space homogeneous eigenvector), we have that $h^{(j)}_1(\omega)(v)$ is also imaginary for all $v \in \RR^d$.
      Gathering this information with the explicit representation \eqref{eqLambdajExplicit}, we obtain that for any $j \in \{-1,\ldots,d\}$, the second order expansion in $\rho$ of $\lambda^{(j)}$, denoted by $\lambda^{(j)}_2$ is nonpositive\footnote{Some explicit computations are given in the $L^2$ case in \cite[Section 4]{ellis:1975}.}.
      The higher order expansions can be computed by the same induction process.
      
      This concludes the proof of Theorem \ref{thmTaylorExpSpectrum}.

	\section*{Acknowledgment}
	  The research of the author was granted by the ERC Starting Grant 2009 \#239983 (NuSiKiMo), NSF Grants  \#1008397 and \#1107444 (KI-Net) and ONR grant \#000141210318.
	  The author would like to thanks F. Filbet and C. Mouhot for their careful reading and fruitful comments on the manuscript.
		
	\begin{appendix}
	  
	  \section{Functional Toolbox on the Collision Operator}
	    \label{secFuncToolbox}
	    
	    Let us present some important properties concerning the granular gases operator we heavily used on this paper.

			To be consistent with \cite{mischler:20091}, we shall define for $\delta > 0$ the regularized operator
			\begin{equation*} 
			  \mathcal L_{1, \delta} \, = \, \mathcal L_{1, \delta}^+ - \mathcal L^* - \mathcal L^\nu,
			\end{equation*}
			where $\mathcal L_{1, \delta}^+$ is the regularization of the truncated gain term introduced in \cite{Mouhot:2006}.
	    One of the key properties of the regularized operator is that it converges towards $\mathcal L_1$ when $\delta \to 0$ in the norm of graph of $L^1(m^{-1})$ (and also in the weighted Sobolev spaces $W^{k,1}_q(m^{-1})$) but with a loss of integration weights:
	    \begin{proposition}[Proposition 5.5 of \cite{mischler:20091}]
	      \label{propApproxL1}
	      For any $k,q \in \NN$, we have
	      \[ \left \| \left (  \mathcal L_{1, \delta} -  \mathcal L_{1}\right )g\right \|_{W^{k,1}_q(m^{-1})} \leq \ve(\delta) \, \|g\|_{W^{k,1}_{q+1}(m^{-1})}, \]
	      where $\ve(\delta)$ is an explicit constant, going to $0$ as $\delta \to 0$.
	    \end{proposition}
	    
	    We then state a result about the H\"older continuity (in the norm of the graph) of the  gain term of the granular gases operator with respect to the restitution coefficient $\alpha$.
	    \begin{proposition}[Proposition 3.2 of \cite{mischler:20091}]
	      \label{propHoldContOp}
	      For any $\alpha, \alpha' \in (0, 1]$, and any $g \in L^1_1(m^{-1})$, $f \in W^{1,1}_1(m^{-1})$, there holds
	      \begin{equation*}
	        \left \{ \begin{aligned}
	          & \left \| \Q_\alpha^+ (g,f) - \Q_{\alpha'}^+ (g,f) \right \|_{L^1(m^{-1})} \leq \ve  \left (\alpha - \alpha'\right ) \|f\|_{W^{1,1}_1(m^{-1})} \| g \|_{L^1_1(m^{-1})}, \\
	          & \left \| \Q_\alpha^+ (f,g) - \Q_{\alpha'}^+ (f,g) \right \|_{L^1(m^{-1})} \leq \ve \left (\alpha - \alpha'\right ) \|f\|_{W^{1,1}_1(m^{-1})} \| g \|_{L^1_1(m^{-1})},
	        \end{aligned} \right.
	      \end{equation*}
	      where we have set 
	      \[\ve(r) = C \,r^{\frac{1}{3+4s}}\] for a constant $s$ given by the weight function $m(v)= \exp \left (- a\, |v|^s\right )$.
	    \end{proposition}

	    We also need to estimate the smoothness, the tail behavior and the pointwise lower bound (uniformly with respect to the restitution coefficient $\alpha$) of the equilibrium profiles $F_\alpha$ solutions to \eqref{eqEquilib}. We have the following result.
	    
	    \begin{proposition}[Propositions 2.1 and 2.3 of \cite{mischler:20092}]
	      \label{propRegSelfSimProf}
	      Let us fix $\alpha_0 \in (0,1)$. There exist some positive constants $a_1, a_2, a_3, a_4$ (independent of $\alpha$) and, for any $k \in \NN$ a positive constant $C_k$ such that for all $\alpha \in [\alpha_0, 1)$
	      \begin{gather*}
	        \|F_\alpha\|_{L^1\left (e^{a_1}|v|\right )} \leq a_2, \quad \|F_\alpha\|_{H^k(\RR^d)} \leq C_k,
 \\
	        F_\alpha(v) \geq a_3 \, e^{-a_4|v|^8}, \quad \forall \, v \in \RR^d.
	      \end{gather*}
	    \end{proposition}
	    
	    Moreover, these profiles converge in $L^1_2$ towards the elastic Maxwellian $F_1$, with an explicit rate:
	    \begin{proposition}[Proposition 3.1 of \cite{mischler:20092}]
	      \label{propCvMaxwProfil}
	      For any $\ve > 0$, there exists $C_\ve$ such that
	      \begin{equation*}
	        \|F_\alpha - F_1\|_{L^1_2} \leq C_\ve (1 - \alpha)^{\frac{1}{2+\ve}}.
	      \end{equation*}
	    \end{proposition}
	    			
			We now define for $\zeta \in \CC$ and $\delta > 0$ the operators
	    \begin{equation}
	      \label{defABops}
	      A_\delta := \mathcal L_{1, \delta}^+ -  \mathcal L^* \quad \text{ and } \quad B_{\alpha, \, \delta} (\zeta) := \LL^{\nu_1} + \mathcal I_\alpha + \zeta - \left ( \mathcal L_{1}^+ - \mathcal L_{1, \delta}^+ \right ),
	    \end{equation}
	    where $\mathcal I_\alpha := \LL_1 - \LL_\alpha$ is the difference between the elastic and inelastic linearized operators.
	    We can then write the problem of computing the inverse of resolvent operator of $\LL_\alpha$ as the perturbation equation
	    \begin{equation*}
	      \LL_{\alpha } - \zeta = A_\delta - B_{\alpha, \, \delta} \left (\zeta \right ).
	    \end{equation*}
	    We  state a result of convergence of the linearized granular gases operator towards the linearized elastic operator (which is a consequence of Proposition \ref{propHoldContOp}), as well as estimates on the operator $B_{\alpha, \, \delta}$.
	   
	    \begin{lemma}[Lemmas 5.9 of \cite{mischler:20091} and 5.2 of \cite{mischler:20092}]
	      \label{lemVracProp}
	      For any $k,q \in \NN$ and any exponential weight function $m$, the following properties hold:
	      \begin{enumerate}[leftmargin=*]
	        \item There exist a constructive $\alpha_0 \in (0,1]$ and some nonnegative constant $C = C(k,q,m)$ such that for any $\alpha \in (\alpha_0, 1]$,
		        \begin{gather*}
		          \|\LL_\alpha \|_{W^{k+2,1}_{q+1}(m^{-1}) \to W^{k,1}_{q}(m^{-1})} \leq C, \\
		          \left \|\LL_\alpha - \LL_1 \right \|_{W^{3,1}_{3}(m^{-1}) \to L^1(m^{-1})} \leq C \, (1 - \alpha).
		        \end{gather*}
	        \item For any $\delta >0$, the operator $A_\delta: L^1 \to W^{\infty,1}_\infty\left (m^{-1}\right )$ is a bounded linear operator (more precisely, it maps function $L^1$ into $ \mathcal{C}^\infty$ functions with compact support).
	        \item There exists some constants $\delta^* >0$ and $\alpha_1 \in (\alpha_0, 1)$ such that for any $\zeta \in \Delta_{-\nu_0}$, $\delta < \delta_*$ and $\alpha \in [\alpha_1,1]$ the operator \[ B_{\alpha, \, \delta} (\zeta) : W^{k+2,1}_{q+1}(m^{-1}) \to W^{k,1}_{q}(m^{-1}) \] is invertible. Moreover, its inverse operator satisfies
		        \begin{gather*}
		          \left \| B_{\alpha, \, \delta} (\zeta)^{-1} \right\|_{W^{k,1}_{q}(m^{-1}) \to W^{k,1}_{q}(m^{-1})} \leq \frac{C_1}{|\nu_0 - \Re e \, \zeta|}, \\
		          \left \| B_{\alpha, \, \delta} (\zeta)^{-1} \right\|_{W^{k,1}_{q}(m^{-1}) \to W^{k+2,1}_{q+1}(m^{-1})} \leq \frac{C_2}{|\nu_0 - \zeta|}
		        \end{gather*}
		        for some explicit constants $C_1, C_2$ depending on $k, q, \delta^*, \alpha_1$.
	      \end{enumerate}
	    \end{lemma}

	    As a consequence of these results, we also have the following proposition.
	    \begin{proposition}[Proposition 3.8 of \cite{mischler:20091}]
	      For any $k,q \in \NN$, any exponential weight function $m$, and  any $\alpha \in (\alpha_0, 1]$,
	      \[ \left \|\LL_\alpha^+ - \LL_1^+ \right \|_{W^{k,1}_{q}(m^{-1}) \to W^{k,1}_{q+1}(m^{-1})} \leq \ve \, (1 - \alpha) \]
	      where $\ve$ has been defined in Proposition \ref{propHoldContOp}.
	    
	    \end{proposition}

	\end{appendix}

	\bibliographystyle{acm}
	\bibliography{/home/Tom/Documents/These/includes/biblio}
	
\end{document}